\documentclass[11pt]{amsart}
\usepackage[british]{babel}

\usepackage{amssymb,latexsym}
\usepackage{comment}
\usepackage{titletoc}
\usepackage{amsfonts, amsmath, amsthm, amssymb, fancybox, graphicx, color, times, babel}
\usepackage[utf8]{inputenc}
\usepackage{multirow}
\usepackage[usenames,dvipsnames]{xcolor}

\usepackage[all,cmtip]{xy}
\usepackage[normalem]{ulem}

\newtheorem {thm}{Theorem}
\newtheorem* {thm*}{Theorem}

\newtheorem {cor}[thm]{Corollary}

\newtheorem* {cor*}{Corollary}
\newtheorem {lem}[thm]{Lemma}

\newtheorem {prop}[thm]{Proposition}
\newtheorem {defi}[thm]{Definition}
\newtheorem {rem}[thm]{Remark}

\theoremstyle{definition}
\newtheorem {exa}[thm]{Example}
\newtheorem* {conj*}{Conjecture}

\newtheorem* {quest*}{Question}

\DeclareMathOperator{\ord}{ord}
\DeclareMathOperator{\Gal}{Gal}

\DeclareMathOperator{\tors}{tors}

\DeclareMathOperator{\Aut}{Aut}

\DeclareMathOperator{\Dens}{Dens}
\DeclareMathOperator{\Res}{Res}

\newcommand{\lval}{v_\ell}

\newcommand{\Q}{\mathbb{Q}}
\newcommand{\Z}{\mathbb{Z}}

\newcommand{\T}{\mathbb{T}}

\newcommand{\p}{\mathfrak{p}}

\DeclareMathOperator{\GL}{GL}
\DeclareMathOperator{\im}{Im}

\newcommand{\tgsp}{\#\mathbb{T}}

\makeatletter
\newcommand{\customlabel}[2]{%
   \protected@write \@auxout {}{\string \newlabel {#1}{{#2}{\thepage}{#2}{#1}{}} }%
   \hypertarget{#1}{}
}
\makeatother

\usepackage{xr-hyper}
\usepackage{hyperref}

\DeclareMathOperator{\Haus}{G}
\DeclareMathOperator{\open}{G}

\DeclareMathOperator{\imGal}{\mathcal G}
\DeclareMathOperator{\ZarGal}{\mathcal G_{Zar}}
\DeclareMathOperator{\ZarGalId}{\mathcal G^0_{Zar}}
\DeclareMathOperator{\ZarGalZ}{\overline{\mathcal G}}
\DeclareMathOperator{\subvar}{V}

\DeclareMathOperator{\algQ}{G}
\DeclareMathOperator{\algZ}{\overline{G}}

\DeclareMathOperator{\algRid}{G}

\DeclareMathOperator{\profi}{G}
\DeclareMathOperator{\profisub}{H}

\DeclareMathOperator{\fine}{F}

\DeclareMathOperator{\failure}{ c_{\text{Kummer}}}

\DeclareMathOperator{\set}{\mathcal W}
\DeclareMathOperator{\function}{w}

\DeclareMathOperator{\neighbourhood}{H}

\title{Reductions of points on algebraic groups}
\begin{document}

\author{Davide~Lombardo and Antonella~Perucca}
\address[]{Leibniz Universit\"at Hannover, Welfengarten 1, 30167 Hannover, Germany}
\email[]{lombardo@math.uni-hannover.de}
\address[]{Universit\"at Regensburg, Universit\"atsstra{\ss}e 31, 93053 Regensburg, Germany}
\email[]{antonella.perucca@mathematik.uni-regensburg.de}


\begin{abstract}
Let $A$ be the product of an abelian variety and a torus defined over a number field $K$. Fix some prime number $\ell$. If $\alpha \in A(K)$ is a point of infinite order, we 
consider the set of primes $\mathfrak p$ of $K$ such that the reduction $(\alpha \bmod \mathfrak p)$ is well-defined and has order coprime to $\ell$. This set admits a natural density. By refining the method of R.~Jones and J.~Rouse (2010), we can express the density as an $\ell$-adic integral without requiring any assumption.
We also prove that the density is always a rational number whose denominator (up to powers of $\ell$) is uniformly bounded in a very strong sense.
For elliptic curves, we describe a strategy for computing the density which covers  every possible case.
\end{abstract}

\maketitle

\section{Introduction}

\subsection{Reductions of a point having order coprime to $\ell$}\label{sect:Intro1}

Let $A$ be a connected commutative algebraic group defined over a number field $K$, and fix some prime number $\ell$. Let $\alpha\in A(K)$ be a point of infinite order and consider the primes $\mathfrak p$ of $K$ for which the reduction of $\alpha$ modulo ${\mathfrak p}$ is well-defined and has order coprime to $\ell$. The aim of this paper is understanding the natural density of this set (provided it exists): 
$$\Dens_\ell(\alpha):=\Dens\{\,\mathfrak p\,:\; \ell\nmid \ord(\alpha \bmod \mathfrak p)\}\,.$$

\subsection{History of the problem} In the sixties, Hasse \cite{Hasse1, Hasse2} considered the case of $A$ being the multiplicative group over the rationals and gave parametric formulas for $\Dens_\ell(\alpha)$. For a survey of related questions for the rational numbers, see \cite{MR3011564} by Moree. 
The second author (partly joint with Debry) extended the method of Hasse to solve the case where $A/K$ is a $1$-dimensional torus over a number field \cite{PeruccaKummer, DebryPerucca, PeruccaTori}.
In \cite{Pink-order}, Pink gave a motivic interpretation of the problem for abelian varieties: considering the tree of $\ell^\infty$ division points over $\alpha$ gives a Tate module $T_{\ell}(A, \alpha)$ which is an extension of the Tate module $T_{\ell}(A)$ by $\Z_{\ell}$ (and is a particular case of the Tate module of 1-motives first described by Deligne in \cite{MR0498552}).

In \cite{JonesRouse}, Jones and Rouse considered the Galois action on the tree of $\ell^\infty$ division points over $\alpha$, which encodes the Kummer representation for $\alpha$ and the $\ell$-adic representation attached to $A$. In \cite[Theorem 3.8]{JonesRouse} they prove -- for any connected commutative algebraic group -- that if the image of the Kummer representation is as large as possible we have
$$\Dens_\ell(\alpha) = \int_{\imGal} \ell^{-\lval(\det(x-I))} \, d\mu_{\imGal} (x)$$
where ${\imGal}$ is the image of the $\ell$-adic representation, identified to a subgroup of a suitable general linear group $\GL_b(\Z_\ell)$, and $d\mu_{\imGal}(x)$ is the normalized Haar measure on ${\imGal}$.
They have also given criteria for their assumptions to be satisfied, and have determined the value of $\Dens_\ell(\alpha)$ for 1-dimensional tori and elliptic curves whenever the images of both the Kummer representation and the $\ell$-adic representation are as large as possible (under a small assumption for CM curves \cite[§5.2]{JonesRouse}).

\subsection{The general formula}

We suppose that the torsion/Kummer extensions $K(A[\ell^n], \ell^{-n}\alpha)$ grow maximally for every sufficiently large $n$ (cf. Definition \ref{defi-conditions}). For the product of an abelian variety and a torus we may assume this condition without loss of generality (cf. Remark \ref{nieuw}). 

In this situation, the natural density $\Dens_\ell(\alpha)$ exists by the argument of \cite[Theorem 3.2]{JonesRouse} and the remark following it. By refining the method of Jones and Rouse we can generalize \cite[Theorem 3.8]{JonesRouse} (which corresponds here to the case $\failure=1$ and $\function\equiv 1$):

\begin{thm}\label{thm-general}
Let $A/K$ be a connected commutative algebraic group defined over a number field, $\alpha\in A(K)$ a point of infinite order and $\ell$ a prime number as in Definition \ref{defi-conditions}. If ${\imGal}$ is the image of the $\ell$-adic representation, we have
\begin{equation}\label{general-formula2}
\Dens_\ell(\alpha)=  \failure\cdot \int_{{\imGal}}\ell^{-\lval(\det(x-I))} \cdot \function(x) \,\, d\mu_{{\imGal}}(x)
\end{equation}
where the constant $\failure:=\failure(A/K,\ell,\alpha)$ is as in Lemma \ref{lemma:DefinitionF} (it measures the failure of maximality for the Kummer extensions) and the function $\function:=\function(A/K,\ell,\alpha)$ is as in Lemma \ref{lemma:Extendw}
($\function(x)$ is either zero or a power of $\ell$ with exponent in $\mathbb Z_{\leqslant 0}$; this function measures a particular relation between the torsion and the Kummer extensions).
\end{thm}

\subsection{Rationality of the density}

For products of abelian varieties and tori, the density is always a rational number (this result is new even for elliptic curves):

\begin{thm}\label{thm:rationality}
If $(A/K,\ell,\alpha)$ are as in Definition \ref{defi-conditions}, $\Dens_\ell(\alpha)$ is a rational number.
\end{thm}

In Section \ref{universality} we even prove (for all products of abelian varieties and tori) that the denominator of $\Dens_\ell(\alpha)$ can be universally bounded up to a power of $\ell$:

\begin{thm}\label{Conjecture}
Fix $g \geqslant 1$. There exists a polynomial $p_g(t) \in \mathbb{Z}[t]$ with the following property: whenever $K$ is a number field and $A/K$ is the product of an abelian variety and a torus with $\dim(A)=g$, then for all prime numbers $\ell$ and for all $\alpha \in A(K)$ we have
\[
\Dens_\ell(\alpha) \cdot p_g(\ell) \in \mathbb{Z}[1/\ell].
\]
\end{thm}

\subsection{The case of elliptic curves}
We have collected in a companion paper \cite{LombardoPerucca1Eig} general results on $\GL_2(\mathbb Z_\ell)$ and all its Cartan subgroups (including the ramified ones): this leads to a very detailed classification of the elements in the image of the $\ell$-adic representation attached to any elliptic curve according to the structure of their group of fixed points in $A[\ell^\infty]$. Using this classification, we prove in Section \ref{EC} the following result:

\begin{thm}\label{compu-thm}
For elliptic curves $\Dens_\ell(\alpha)$ can be effectively computed. Furthermore, we have $\Dens_\ell(\alpha)\cdot (\ell-1)(\ell^2-1)^2(\ell^E-1) \in \mathbb{Z}[1/\ell]$, where $E=4$ if the elliptic curve has complex multiplication over $\overline{K}$ and $E=6$ otherwise.
\end{thm}

\subsection{Further results} In Section \ref{interpretation} we give a cohomological interpretation of the density, and more precisely of \cite[Theorem 3.2]{JonesRouse}:

\begin{thm}\label{thm:interpretation}
If $(A/K,\ell,\alpha)$ are as in Definition \ref{defi-conditions}, $\Dens_\ell(\alpha)$ equals the Haar measure in $\Gal(K(A[\ell^\infty], \ell^{-\infty}\alpha)/K)$ of the set of automorphisms $\sigma$
such that the Kummer cohomology class of $\alpha$ (defined in Section \ref{sec-Kummerclass}) is in the kernel of the restriction map $$\Res_\sigma:  H^1\big(\Gal(K(A[\ell^\infty], \ell^{-\infty}\alpha)/K), T_\ell A\big) \rightarrow H^1({\langle \sigma \rangle}, T_\ell A)\,,$$
where $\langle \sigma \rangle$ denotes the procyclic subgroup generated by $\sigma$.
\end{thm}

The value of $\Dens_\ell(\alpha)$ is indeed only related to the field $K(A[\ell^\infty], \ell^{-\infty}\alpha)$, see Proposition \ref{alien}.
Finally, in Section \ref{sect:asymptotics} we show that (when $A$ is the product of an abelian variety and a torus) $\Dens_\ell(\ell^n\alpha)$ converges to $1$ as $\ell^n$ tends to infinity, and this uniformly in the choice of $\alpha \in A(K)$.

\subsection{Notation} 
If $\Haus$ is a compact Hausdorff topological group, we denote by $\mu_{\Haus}$ (or simply by $\mu$) its normalized Haar measure.
We denote by $\ell$ a fixed prime number and call $v_\ell$ the $\ell$-adic valuation on $\mathbb{Q}_\ell$ (we set $v_\ell(0)=+\infty$ and $\ell^{-\lval(0)}=0$).
We write $\operatorname{Mat}_b(\mathbb{Z}_\ell)$ for the ring of $b \times b$ matrices with coefficients in $\mathbb{Z}_\ell$, and define  analogously $\operatorname{Mat}_b(\mathbb{Z}/\ell^n\mathbb{Z})$. We denote by $I$ the identity matrix/endomorphism. The $\ell$-adic valuation of a matrix is the minimum of the valuations of its entries, and we write $\det_{\ell}$ for the $\ell$-adic valuation of the  determinant.
Because of its frequent use, we reserve the notation $\imGal$ for the image of the $\ell$-adic representation.

\subsection*{Acknowledgements} 
We thank R.\@ Jones, J.\@ Rouse, P.\@ Jossen and A.\@ Sutherland for helpful discussions. The second author gratefully acknowledges funding from the SFB-Higher Invariants at the University of Regensburg.

\section{Torsion fields and Kummer extensions}\label{sect:Arboreal}

\subsection{The torsion, Kummer, and arboreal representations}
We recall from \cite{JonesRouse} the construction of the {arboreal representation} attached to $A/K$ and to a point $\alpha \in A(K)$, which describes the natural Galois action on the tree of division points over $\alpha$.

If $A$ is a connected commutative algebraic group, we define its Tate module $T_\ell A$ as the projective limit of the torsion groups $A[\ell^n]$ (the transition homomorphism $A[\ell^{n+1}]\rightarrow A[\ell^n]$ is multiplication by $\ell$). The Tate module is a $\Z_\ell$-module isomorphic to $\Z_\ell^b$, where $b$ is the first Betti number of $A$ (for elliptic curves $b=2$).

The \emph{torsion} (or \emph{$\ell$-adic}) \emph{representation} of $A$ is the representation of $\Gal(\overline{K}/K)$ with values in the automorphism group of $T_\ell A$ which is induced by the natural Galois action on the torsion points of $A$. Choosing a $\Z_\ell$-basis for $T_\ell A$ means fixing an isomorphism of $T_\ell A$ with $\Z_\ell^b$, so the choice of a basis allows us to identify the image of the $\ell$-adic representation with a subgroup of $\GL_b(\Z_\ell)$.
We also consider the \emph{mod $\ell^n$ representation}, whose image is a subgroup of $\GL_b(\Z/\ell^n\Z)$ and which describes the Galois action on $A[\ell^n]$.

The \emph{Kummer representation} depends both on $A/K$ and the point $\alpha$. For every $n\geqslant 1$ we call $\ell^{-n}(\alpha)$ the set of points in $A(\overline{K})$ whose $\ell^n$-th multiple equals $\alpha$. The fields
$$K_{n}:=K(A[\ell^n])\qquad \text{and}\qquad K_{\alpha,n}:=K_n(\ell^{-n}(\alpha))$$ are then finite Galois extensions of $K$. We denote the countable union of these fields by $K_{\infty}$ and $K_\alpha$ respectively.
We then define the {Kummer map} as
\begin{equation}\label{Kummer-rep}
\begin{array}{ccc}
 \Gal(K_\alpha/K_\infty) & \to & T_\ell A \\
  \sigma & \mapsto & \left( \sigma(\beta_n)-\beta_n \right)_{n \geqslant 1}
\end{array}
\end{equation}
where $\{\beta_n\}_{n \geqslant 1}$ is any sequence of points $\beta_n \in A(\overline{K})$ satisfying $\ell \beta_1=\alpha$ and $\ell \beta_{n+1}=\beta_n$ for all $n \geqslant 1$.
The definition does not depend on the choice of the sequence because $\sigma$ is the identity on $K_\infty$.

The \emph{arboreal representation} encodes both the $\ell$-adic representation and the Kummer representation, being the map
\begin{equation}\label{eq:ArborealRepresentation}
\begin{array}{cccc}
\omega : & \Gal(K_{\alpha}/K) & \to & T_\ell A \rtimes \Aut(T_\ell A)\\
& \sigma & \mapsto & (t_{\sigma}, M_{\sigma})
\end{array}
\end{equation}
where $M_{\sigma}$ is the image of $\sigma$ (more precisely, of any lift of $\sigma$ to $\Gal(\overline{K}/K)$) under the $\ell$-adic representation and where we define
$t_{\sigma}:= \left( \sigma(\beta_n)-\beta_n \right)_{n \geqslant 1}$ for some fixed choice of $\{\beta_n\}_{n \geqslant 1}$ as above. The arboreal representation is an injective group homomorphism that identifies $\Gal(K_{\alpha}/K)$ to a subgroup of $T_\ell A \rtimes \Aut(T_\ell A)\cong \Z_\ell^b \rtimes \GL_b(\Z_\ell)$. With this identification, we shall write 
\begin{equation}\label{notation-sigma}
\sigma=(t_{\sigma}, M_{\sigma}).
\end{equation}
We employ the same notation for $\sigma \in \Gal(K_{\alpha,n}/K)$, in which case  we have $t_\sigma \in A[\ell^n] \cong (\Z/\ell^n\Z)^b$ and $M_\sigma \in \Aut A[\ell^n] \cong \GL_b(\Z/\ell^n\Z)$. 

\subsection{Growth conditions for torsion fields and Kummer extensions}\label{torsionkummer}
We denote by $\imGal \subseteq \GL_{b}(\Z_\ell)$ the image of the $\ell$-adic representation and by $\imGal(n)$ the image of the mod $\ell^n$ representation, i.e.\@ the image of $\imGal$ under the natural projection $\GL_{b}(\mathbb{Z}_\ell) \to \GL_b(\Z/\ell^n\Z)$.
We write $\dim \imGal$ for the dimension of the  Zariski closure of $\imGal$ in $\GL_{b,\mathbb{Q}_\ell}$.

If $A$ is an elliptic curve, by work of Serre \cite{Serre} and by the classical theory of complex multiplication \cite{MR0236190} we know the following: if $\operatorname{End}_{\overline{K}}(A)=\mathbb{Z}$, then $\imGal$ is an open subgroup of $\GL_2(\Z_\ell)$ ($\dim \imGal=4$), and otherwise $\imGal$ is open in the normalizer of a Cartan subgroup of $\GL_2(\Z_\ell)$ ($\dim \imGal =2$). See \cite{LombardoPerucca1Eig} for a classification of all Cartan subgroups and their normalizers.

\begin{defi}\label{defi-conditions} We say that $(A/K,\ell)$ satisfy the  
\emph{eventual maximal growth of the torsion fields} if there exists a positive integer $n_0$ such that we have 
\begin{equation}\label{C1}
\text{$\#\imGal(n+1)/\#\imGal(n)=\ell^{\dim \imGal}$\qquad for every $n\geqslant n_0$ .}
\tag{C1}
\end{equation}
We say that $(A/K,\ell, \alpha)$  satisfy the \emph{eventual maximal growth of the Kummer extensions} if there exists a positive integer $n_0$ such that we have:
\customlabel{C2}{C2}
\begin{equation}\label{C2i}
\text{$K_{n, \alpha}$ and $K_{n'}$ are linearly disjoint over $K_{n}$\qquad for every $n'\geqslant n \geqslant n_0$} \tag{C2i}
\end{equation}
\begin{equation}
\text{$[K_{n'}(\ell^{-n'}\alpha):K_{n'}(\ell^{-n}\alpha)]=\ell^{b(n'-n)}$\qquad for every $n'\geqslant n \geqslant n_0$\,.} \tag{C2ii}\label{C2ii}
\end{equation}
\end{defi}

Equivalently, there exists a positive integer $n_0$ such that we have
\begin{equation}\label{refomulation-Defi}
[K_{n'}(\ell^{-n'}\alpha):K_{n}(\ell^{-n}\alpha)]=(\ell^{b+\dim \imGal})^{n'-n}\qquad \text{for every $n'\geqslant n \geqslant n_0$}\,.
\end{equation}

\subsection{Results on the growth conditions}

\begin{rem}\label{nieuw}
Let $A/K$ be the product of an abelian variety and a torus. Then $A$ satisfies Condition \eqref{C1} for any prime $\ell$; moreover, it also satisfies Conditions \eqref{C2i} and \eqref{C2ii} for any $\alpha \in A(K)$ such that $\mathbb Z \alpha$ is Zariski-dense in $A$.
These facts follow from Lemma \ref{Gisp} below and \cite[Theorem 2]{Bertrand}. Thus if $A/K$ is the product of an abelian variety and a torus and $\alpha$ has infinite order, we may always reduce to the situation of Definition \ref{defi-conditions} by \cite[Main Theorem]{MR2473894}. Indeed, consider the smallest algebraic subgroup $A'$ of $A$ containing $\alpha$. If the number $n$ of connected components of $A'$ is divisible by $\ell$,  
then we have $\Dens_\ell(\alpha)=0$. Otherwise we may replace $\alpha$ by $[n]\alpha$ and hence work with the connected component of the identity of $A'$ in place of $A$. In this case we have $\Dens_\ell(\alpha)>0$.
\end{rem}

\begin{lem}\label{lemma:DefinitionF}
If \eqref{C2i}-\eqref{C2ii} hold, the following rational number is independent of $n$  for $n \geqslant n_0$:
\begin{equation}\label{defi-F}
\failure:=\ell^{bn}/\#\Gal(K_{\alpha,n}/K_{n})\,.
\end{equation}
\end{lem}
\begin{proof}
$\#\Gal(K_{\alpha,n}/K_{n}) \cdot \ell^{-b(n-n_0)}= \#\Gal(K_{n}(\ell^{-n_0}\alpha)/K_{n})=\#\Gal(K_{\alpha,n_0}/K_{n_0})\,.$ 
\end{proof}

\begin{lem}\label{matrici}
Let $\algQ$ be an algebraic subgroup of $\GL_{b, \mathbb{Q}_\ell}$. Define $\algZ:=\algQ(\mathbb{Q}_\ell) \cap \GL_b(\Z_\ell)$ and write $\algZ(n)$ for the reduction modulo $\ell^n$ of $\algZ$. The sequence $\#\algZ(n+1)/\#\algZ(n)$ is non-decreasing for $n\geqslant 2$ and it is eventually equal to $\ell^{\dim \algQ}$.
\end{lem}

\begin{proof}
We have to study the order of $\mathrm{Ker}(n)$, the kernel of the reduction map $\algZ(n+1) \to \algZ(n)$. 
For every $n$, the map $M \mapsto \ell^{-n}(M-I)$ gives a group isomorphism between $\mathrm{Ker}(n)$ and some vector subspace $V_n\subseteq \operatorname{Mat}_b(\Z/\ell\Z)$. 
The sequence $\#V_n$ is bounded from above by $\#\operatorname{Mat}_b(\Z/\ell\Z)$, and  it is non-decreasing: indeed, we now show that $V_n \subseteq V_{n+1}$. If $v\in V_n$, then $I+\ell^{n} v\in \mathrm{Ker}(n) \subseteq \algZ(n+1)$, so there is some $\tilde{M}:=I+\ell^{n} \tilde{v}$ in $\algZ$ that is congruent to $1+\ell^n v$ modulo $\ell^{n+1}$. We have $v\in V_{n+1}$ because $\mathrm{Ker}(n+1)$ contains $I+\ell^{n+1} v$: indeed, this is the image in $\algZ(n+2)$ of $\tilde{M}^\ell$ (this follows from $\ell^{n+2}\mid \ell^{2n}$). This proves that $\#V_n$ is eventually constant, and since the sequence $\#\algZ(n) \ell^{-n\dim \algQ}$ converges to some positive number by \cite[Theorem 2]{MR656627} we must have $\#V_n=\ell^{\dim \algQ}$ for all $n$ sufficiently large.
\end{proof}

\begin{lem}\label{Gisp}
Semiabelian varieties satisfy \eqref{C1} for any prime $\ell$.
\end{lem}
\begin{proof}
Let $\ZarGal$ be the Zariski closure of $\imGal$ in $\GL_{b, \mathbb{Q}_\ell}$ and define $\ZarGalZ:=\ZarGal (\Q_\ell) \cap \GL_b(\Z_\ell)$.
By Proposition \ref{prop:GeneralizedBogomolov} we know that $\imGal$ is open in $\ZarGalZ$, so there exists some positive integer $n_0$ such that for every $n\geqslant n_0$ the matrices of $\ZarGalZ$ that reduce to the identity modulo $\ell^{n}$ are in $\imGal$. The statement then follows from Lemma \ref{matrici}, because for every sufficiently large $n$ we have $\ker\left(\imGal(n+1) \to \imGal(n) \right)=\ker\left( \ZarGalZ(n+1) \to \ZarGalZ(n) \right)$.
\end{proof}

The proof of Lemma  \ref{matrici} implies that the following definition is well-posed:
\begin{defi}\label{def:TgSpace}
Let $\open$ be a subgroup of $\operatorname{GL}_b(\mathbb{Z}_\ell)$ that is open (for the $\ell$-adic topology) in its Zariski closure. The image of the map
\[
\begin{array}{ccc}
\ker(\open(n+1) \to \open(n)) & \to & \operatorname{Mat}_{b}(\mathbb Z/\ell\Z) \\
M & \mapsto & \ell^{-n}(M-I)\end{array}
\]
is independent of $n$ for all sufficiently large $n$: it is a vector space of the same dimension as the Zariski closure of $\open$, and we call it the \textit{tangent space} $\T$ of $\open$.
\end{defi}

\subsection{Effectivity of Definition \ref{defi-conditions}}

Consider the arboreal representation $\omega$ as in \eqref{eq:ArborealRepresentation} and its reduction
$$\omega_{n} : \Gal(K_{\alpha}/K) \to A[\ell^{n}] \rtimes \Aut(A[\ell^{n}])\,.$$

\begin{thm}\label{thm:IfXThenLevel}
Let $n \geqslant 1$ (resp.~$n \geqslant 2$ if $\ell=2$).
\begin{enumerate}
\item[(i)] If $[K_{n+1}:K_{n}]=\tgsp$ holds, we have $[K_{m+1} : K_{m}]=\tgsp$ for all $m \geqslant n$.
\item[(ii)] If $[K_{\alpha,n+1}:K_{\alpha,n}]=\tgsp\ell^b$ holds, we have $[K_{\alpha,m+1} : K_{\alpha,m}]=\tgsp\ell^b$ for all $m \geqslant n$ and the image of $\omega$ is the inverse image in $T_\ell A \rtimes \Aut(T_\ell(A))$ of the image of $\omega_{n}$.
\end{enumerate}
\end{thm}

\begin{proof}
(i) Define $H_{m}:= \ker(\imGal(m+1) \to \imGal(m))$. We know $\#H_{n}=\tgsp$ and by induction we prove $\#H_{m}=\tgsp$ for $m\geqslant n$. Write the elements of $H_{m}$ as $I+\ell^{m}x$, where $x$ varies in a subset of $\mathrm{Mat}(\Z/\ell\Z)$ of cardinality $\tgsp$. To prove $\#H_{m+1}=\tgsp$, we show that $I+\ell^{m+1}x$ is in $\imGal(m+2)$: this group contains $I+\ell^{m}x'$, where $x'$ is some lift of $x$ to $\mathrm{Mat}(\Z/\ell^2\Z)$, and hence also $(I+\ell^{m}x')^\ell=I+\ell^{m+1}x$.

(ii) The kernel of the projection $\Gal(K_{\alpha,m+1}/K) \to \Gal(K_{\alpha,m}/K)$ is $\Gal(K_{\alpha,m+1}/K) \cap H'_m$, where we set $H'_m:=A[\ell]\rtimes H_m$. From (i) we know that $\#H'_{m}= \tgsp\ell^b$, so it suffices to prove $H'_{m}\subseteq \Gal(K_{\alpha,m+1}/K)$. We know this assertion for $m=n$, so we suppose that it holds for some $m\geqslant n$ and prove it for $m+1$.
Since $H'_{m+1}$ is generated by $A[\ell] \times \{I\}$ and by $\{0\} \times H_{m+1}$, it suffices to prove that these are contained in $\Gal(K_{\alpha,m+2}/K)$.
For $t \in A[\ell]$, we have $(t,I)\in H'_m\subseteq \Gal(K_{\alpha,m+1}/K)$, so there is $(u,M)\in \Gal(K_{\alpha,m+2}/K)$ satisfying $[\ell] u = t$ and $M \equiv I  \pmod {\ell^{m+1}}$; we have $(u, M)^\ell  = (t, I)$ because more generally $(u,M)^{k}=([k]u,M^k)$ holds by induction for $k \geqslant 1$.\\
Writing again an element of $H_{m+1}$ as $h=I+\ell^{m+1}x$, we know that $(0,I+\ell^{m}x)$ is in $\{0\} \times H_{m}\subseteq\Gal(K_{\alpha,m+1}/K)$ so we deduce as above that $\Gal(K_{\alpha,m+2})$ contains $(t,h)$ for some $t\in A[\ell]$ and we conclude because $(0,h)=(-t, I)(t,h)$.

For every $m\geqslant n$ we have proven that $\im(\omega_{m+1})$ is the inverse image of $\im(\omega_{m})$ in $A[\ell^{m+1}] \rtimes \Aut(A[\ell^{m+1}])$, so we conclude by taking the limit in $m$.
\end{proof}

\begin{rem}\label{rem:n0ComputableForEllipticCurves}
Let $(A/K,\ell,\alpha)$ be as in Definition \ref{defi-conditions}. 
Provided that $\dim \imGal$ is known, by Theorem \ref{thm:IfXThenLevel} we may take for $n_0$ the smallest integer $n\geqslant 1$ ($n\geqslant 2$ for $\ell=2$) satisfying $[K_{n+1}(\ell^{-(n+1)}\alpha):K_{n}(\ell^{-n}\alpha)]=\ell^{b+\dim \imGal}$.
Recall that the problem of determining the Galois group of a number field can be effectively solved, and that the fields $K_{m}(\ell^{-m}\alpha)$ are generated over $K$ by the roots of some explicit division polynomials, thus the above condition can be effectively tested.
If $A$ is an elliptic curve, $\dim \imGal$ is either $2$ or $4$ (see Section \ref{torsionkummer}) and one can algorithmically decide which case applies \cite{MR2181871}, so for elliptic curves the parameter $n_0$ is effectively computable.
\end{rem}

\subsection{Auxiliary results}

\begin{lem}[{\cite[Lemma 4.4]{MR1981599}}]\label{lem:UnramifiedAlmostEverywhere}
Let $A$ be a connected commutative algebraic group defined over a number field $K$. For any prime $\ell$, the $\operatorname{Gal}(\overline{K}/K)$-representation afforded by $T_\ell(A)$ is unramified almost everywhere.
\end{lem}

\begin{prop}\label{finiteness-Av}
If $A/K$ is a connected commutative algebraic group defined over a number field, the torsion subgroup of $A(K_\mathfrak p)$ is finite for every prime $\mathfrak p$ of $K$ (where $K_\mathfrak p$ denotes the completion of $K$ at $\mathfrak{p}$).
\end{prop}
\begin{proof}
Consider $A/K_\mathfrak{p}$ and its Chevalley decomposition $1 \to A_1 \to A \to A_2 {\to} 1$, where $A_1$ is a connected commutative linear algebraic group and $A_2$ is an abelian variety. Since $K_\mathfrak{p}$ is a $p$-adic field, the torsion subgroup of $A_2(K_\mathfrak p)$ is finite by a classical theorem of Mattuck \cite{Mattuck}, so it suffices to show that the intersection between $\ker(A(K_\mathfrak{p}) \to A_2(K_\mathfrak{p}))$ and the torsion subgroup of $A(K_\mathfrak p)$ is finite: we are thus reduced to proving that the torsion subgroup of $A_1(K_\mathfrak{p})$ is finite.

The solvable group $A_1$ has a normal unipotent subgroup such that the quotient is of multiplicative type (and connected, hence a torus): by the same argument as above, it suffices to treat the case of unipotent groups and of tori separately. By the Lie-Kolchin Theorem (and since we are in characteristic zero) the unipotent subgroup is torsion-free, while the assertion is clear for tori. \end{proof}

\begin{prop}\label{prop:GeneralizedBogomolov}
For a semiabelian variety defined over a number field, the image of the $\ell$-adic representation is open in its Zariski closure.\end{prop}
\begin{proof}
We mimic the argument for abelian varieties \cite{MR587337, MR574307}. As explained in \cite[Proof of Lemma 1]{MR587337}, it suffices to show that the Lie algebra of the $\ell$-adic representation $\rho_\ell$ is algebraic. By \cite[Lemma 1 and Corollary to Lemma 2]{MR587337}, it suffices to prove:
\begin{enumerate}
\item $\rho_\ell$ is unramified outside a finite set of primes of $K$;
\item $\rho_\ell$ is of Hodge-Tate type for all the primes of $K$ lying over $\ell$;
\item elements that act semisimply on $T_\ell(A) \otimes_{\Z_\ell} \mathbb{Q}_\ell$ are dense in the image of $\rho_\ell$.
\end{enumerate} 
Property (1) is true by Lemma \ref{lem:UnramifiedAlmostEverywhere}.
The existence of the Hodge-Tate decomposition was proved by Faltings in \cite{MR924705} for smooth {proper} varieties, and it also holds for semiabelian varieties because of the existence of good compactifications for quasi-projective varieties, see \cite{MR2904571} and \cite[Remark 1.2]{2016arXiv160601921S}. The last property holds because the Frobenius automorphisms act semisimply on the rational Tate module (this follows easily from the analogous statement for abelian varieties and for tori because the two sets of Frobenius eigenvalues are disjoint).
\end{proof}

\section{Cohomological interpretation of the density}\label{sect:CohomologyInterpretation}

\subsection{The Kummer cohomology class}\label{sec-Kummerclass}
Let $A$ be a connected commutative algebraic group defined over a number field $K$ and let $\ell$ be a prime number.
If $\alpha\in A(K)$ and $n$ is a positive integer, we denote by $\ell^{-n}(\alpha)$ the set of points $\alpha'\in A(\overline{K})$ satisfying $[\ell^n]\alpha'=\alpha$. We then call $\ell^{-\infty}(\alpha)$ the set consisting of all sequences $\beta:=\{\beta_n\}_{n\geqslant 1}$ satisfying 
$$[\ell]\beta_1=\alpha \qquad  \text{and} \qquad [\ell]\beta_{n+1}=\beta_n\quad \forall n\geqslant 1\,.$$
We have $\ell^{-n}(0)=A[\ell^n]$ and $\ell^{-\infty}(0)=T_\ell A$.
If $\beta, \beta'$ are in $\ell^{-\infty}(\alpha)$ and $\sigma\in \Gal(\overline{K}/K)$, we define $\sigma(\beta):=\{\sigma(\beta_n)\}_{n\geqslant 1}$ and $\beta'- \beta:=\{\beta'_n- \beta_n\}_{n\geqslant 1}$. So for any $\beta \in \ell^{-\infty}(\alpha)$ we get a cocycle
$$
\begin{array}{cccc}
c_\beta: & \Gal(\overline{K}/K)& \rightarrow & T_\ell A\\ & \sigma & \mapsto & \sigma(\beta)-\beta.
\end{array}$$
The induced map from $\Gal(K_\alpha/K)$ agrees with \eqref{Kummer-rep} on $\Gal(K_\alpha/K_\infty)$. A different choice of $\beta$ alters $c_\beta$ by a coboundary, so its class $C_\alpha$ in 
$H^1(\Gal(\overline{K}/K), T_\ell A)$ is well-defined: we call it the \emph{Kummer class} of $\alpha$. We equivalently consider $C_\alpha$ to be in $H^1(\Gal({K_\alpha}/K), T_\ell A)$ and  denote by  $C_{\alpha,n}$ its image in $H^1(\Gal({K_\alpha}/K), A[\ell^n])$, which is obtained by replacing a sequence with its term of index $n$. Notice that $C_{\alpha,n}$ is trivial if and only if there is some point in $\ell^{-n}(\alpha)$ which is defined over $K$.

\subsection{Cohomological conditions}\label{interpretation}
For $\sigma\in \Gal({K_\alpha}/K)$ the restriction map with respect to the profinite cyclic subgroup generated by $\sigma$ is
$$\Res_\sigma:  H^1(\Gal({K}_\alpha/K), T_\ell A)\rightarrow H^1({\langle\sigma\rangle}, T_\ell A)\,.$$
Likewise, for $\tau\in\Gal(K_{\alpha,n}/K)$ the restriction map with respect to the cyclic subgroup generated by $\tau$ is
$$\Res_{\tau}:  H^1(\Gal(K_{\alpha,n}/K), A[\ell^n])\rightarrow H^1({\langle\tau\rangle}, A[\ell^n])\,.$$
Consider the following sets:
$$S_\alpha:=\{\sigma\,:\, C_\alpha\in \ker(\Res_\sigma) \}=\{\sigma\;:\;  \sigma{\beta}=\beta\, \text{for some $\beta\in \ell^{-\infty} \alpha$}\}\subseteq \Gal(K_{\alpha}/K)$$
$$S_{\alpha,n}:=\{\tau\,:\, C_{\alpha,n}\in \ker(\Res_\tau) \}=\{\tau\;:  \tau{\beta_n}=\beta_n\, \text{for some $\beta_n\in \ell^{-n} \alpha$}\}\subseteq \Gal(K_{\alpha,n}/K)\,.$$

Suppose that \eqref{C2i} and \eqref{C2ii} of Definition \ref{defi-conditions} hold and consider $n\geqslant n_0$: if $\tau \in S_{\alpha,n}$ fixes $\beta_n \in \ell^{-n} \alpha$,
then there is $\tau' \in S_{\alpha,n+1}$ over $\tau$ that fixes some $\beta_{n+1} \in \ell^{-(n+1)} \alpha$ satisfying $[\ell]\beta_{n+1}=\beta_n$. We deduce that $S_{\alpha,n}$ is the image of $S_{\alpha,n+1}$ (by passage to the limit, also of $S_{\alpha}$) in $\Gal(K_{\alpha,n}/K)$.
Thus the Haar measure of $S_\alpha$ in $\Gal(K_{\alpha}/K)$ is well-defined and its value is 
$$\mu(S_\alpha)=\lim_{n\rightarrow \infty} \frac{\#S_{\alpha,n}}{\# \Gal({K_{\alpha,n}/K)}}$$
because we take the limit of a non-increasing sequence of positive numbers.

Even though \cite[Theorem 3.2]{JonesRouse} is stated only for products of abelian varieties and tori, the proof works equally well if one just assumes that the triple $(A/K, \ell, \alpha)$ satisfies the conditions of Definition \ref{defi-conditions}.
We then recover Theorem \ref{thm:interpretation} in the form
\begin{equation}\label{eq:interpretation}
\Dens_\ell(\alpha)=\mu(S_\alpha)\,.
\end{equation}

A similar result holds for the density of reductions such that the $\ell$-adic valuation of the order of $(\alpha \bmod \mathfrak p)$ is at most $n$: the cohomological condition becomes $C_\alpha\in \ker([\ell^n] \Res_\sigma)$.

\subsection{Remarks}

\begin{rem}[{\cite[Proof of Theorem 3.8]{JonesRouse}}]\label{ideal}
Writing $\sigma=(t_{\sigma}, M_{\sigma})\in \Gal(K_{\alpha,n}/K)$ as in \eqref{notation-sigma}, we have 
$$\sigma\in S_{\alpha,n}\qquad \Leftrightarrow\qquad t_\sigma\in \im(M_\sigma-I)\,.$$ Indeed, if $t_\sigma=(M_\sigma-I)\gamma$ for some $\gamma\in A[\ell^n]$ then we have $t_\sigma=\sigma(\beta_n)-\beta_n=\sigma(\gamma)-\gamma$ and hence $\beta_n-\gamma$ is a point in $\ell^{-n}(\alpha)$  fixed by $\sigma$. Conversely, if some $\beta'_n\in \ell^{-n}(\alpha)$ is fixed by $\sigma$ then $\beta_n-\beta'_n$ is in $A[\ell^n]$ and its image under $M_\sigma-I$ is $\sigma(\beta_n)-\beta_n=t_\sigma$. The same remark holds for $S_{\alpha}$, hence we have
\begin{equation}\label{bis}
S_\alpha=\{\sigma=(t,M)\in \Gal(K_{\alpha}/K)\;:\; M\in \imGal\text{ and }  t\in \im (M-I)\}\,.
\end{equation}
\end{rem}

\begin{lem}\label{lemma:Haar}\label{Haar2}{\cite[Lemma 18.1.1 and Proposition 18.2.2]{book:71486}}
Let $\profi$ be a profinite group and $\profisub$ a closed normal subgroup of $\profi$. If $\pi$ denotes the natural projection $\profi \to  \profi/\profisub$, then for any $S\subseteq \profi/\profisub$ the preimage $\pi^{-1}(S)$ is measurable in $\profi$, and its Haar measure is $\mu_{\profi/\profisub}(S)$. If $\profi/\profisub$ is finite (i.e.\@ if $\profisub$ is open), this measure  equals $\#S/\#( \profi/\profisub)$.
\end{lem}

\begin{rem}\label{rmk:SalphaSubsetOfTheAbsoluteGaloisGroup}
We may equivalently consider $S_\alpha$ as a subset of $\Gal(\overline{K}/K)$ or of $\Gal(K_\alpha/K)$: since $\bar{\sigma}\in\Gal(\overline{K}/K)$ acts on $\ell^{-\infty}(\alpha)$ through its image $\sigma\in\Gal(K_\alpha/K)$, the set
\[
\overline{S_\alpha}=\{ \bar{\sigma} \in \Gal(\overline{K}/K) : \Res_{\bar{\sigma}}(C_\alpha)=0 \}
\]
is the inverse image in $\Gal(\overline{K}/K)$ of $S_\alpha$, and hence $\mu_{\Gal(\overline{K}/K)}(\overline{S_\alpha})=\mu_{\Gal(K_\alpha/K)}(S_\alpha)$ by Lemma \ref{Haar2}. \end{rem}

\begin{prop}\label{alien}
If $L/K$ is any Galois extension which is linearly disjoint from $K_{\alpha}$ over $K$ we have $\Dens_L(\alpha)=\Dens_K(\alpha)$.
\end{prop}
\begin{proof}
By Theorem \ref{thm:interpretation} and Remark \ref{rmk:SalphaSubsetOfTheAbsoluteGaloisGroup} we have to prove that, considering $S_\alpha$ as a subset of $\Gal(K_\alpha/K)$ or of $\Gal(L_\alpha/L)$, we have $\mu_{\Gal(K_\alpha/K)}(S_\alpha)=\mu_{\Gal(L_\alpha/L)}(S_\alpha)$. 
Since $L$ and $K_\alpha$ are linearly disjoint over $K$, the restriction map $\Gal(L_\alpha/L) \to \Gal(K_\alpha/K)$ is an isomorphism of groups and of measured spaces. We may easily conclude because $\ell^{-\infty}(\alpha)\subset K_\alpha$, so that in particular the action of $\operatorname{Gal}(L_\alpha/L)$ on $\ell^{-\infty} \alpha$ factors through $\Gal(K_\alpha/K)$.
\end{proof}

\section{The density as an $\ell$-adic integral}\label{sect:MeasurePreimages}

\subsection{The $1$-Eigenspace for elements in the image of the $\ell$-adic representation}\label{sect:1Eigen}
Recall that we denote by $\imGal$ the image of the $\ell$-adic Galois representation attached to $A$. 
For every $M \in \operatorname{Aut}(T_\ell(A))$, the kernel of $M- I : A[\ell^\infty] \to A[\ell^\infty]$ is a (possibly infinite) abelian $\ell$-group. We restrict our attention to those $M$ for which $\ker(M-I)$ is finite. If $\fine$ is a finite abelian $\ell$-group with at most $b$ cyclic components, we define  
\begin{equation}\label{MGH}
\mathcal M_{\fine}:=\{M\in \imGal: \, \ker \left (M-I : A[\ell^\infty] \to A[\ell^\infty]\right) \cong  \fine  \}
\end{equation}
and also define
\begin{equation}\label{MathcalM}
\mathcal M :=\bigcup_{\fine} \mathcal{M}_{\fine},
\end{equation}
where the union is taken over all finite abelian $\ell$-groups with at most $b$ cyclic components.
We write $\mathcal{M}_{\fine}(n)$ for the image of $\mathcal{M}_{\fine}$ under the reduction map $\imGal \to \imGal(n)$, and denote by $\exp \fine$ the exponent of the finite group $\fine$.

\begin{lem}\label{lemma:EverythingHasWellDefinedAB}
The set $\mathcal{M}_{\fine}$ of \eqref{MGH} is measurable in $\imGal$ and we have $\mu (\mathcal{M}_{\fine})=\mu(\mathcal{M}_{\fine}(n))$ for every  $n>v_\ell(\exp \fine)$. In particular we have $\mu (\mathcal{M}_{\fine})=0$ if and only if $\mathcal{M}_{\fine}=\emptyset$. The set $\mathcal M$ of \eqref{MathcalM} is measurable in $\imGal$ and, if $A$ satisfies \eqref{C1}, we have $\mu(\mathcal{M})=1$.
\end{lem}
\begin{proof}
Call $\pi_n: \imGal\rightarrow \imGal(n)$ the reduction modulo $\ell^n$.  
For $n>v_\ell(\exp \fine)$ the defining condition for $\mathcal M_{\fine}$ can be checked modulo $\ell^n$, so we have $\mathcal{M}_{\fine}=\pi_n^{-1}(\mathcal{M}_{\fine}(n))$ and the first assertion follows from Lemma \ref{lemma:Haar}. The set $\mathcal{M}$ is measurable because it is a countable union of measurable sets, and we are left to prove $\mu(\imGal\setminus \mathcal M)=0$. Since $\imGal\setminus \mathcal M \subseteq \pi_n^{-1}\left( \pi_n\left(\imGal\setminus \mathcal M \right) \right)$,  by Lemma \ref{lemma:Haar} it suffices to show that
\begin{equation}\label{eq:HaarSmall}
\mu(\pi_n\left(\imGal\setminus \mathcal M \right))=\frac{\# \pi_n\left(\imGal\setminus \mathcal M \right)}{\#\imGal(n)}
\end{equation}
tends to 0 as $n$ tends to infinity. By \eqref{C1}, the cardinality of $\imGal(n)$ is asymptotically given by a constant (positive) multiple of $\ell^{n \dim \imGal}$. Let $\ZarGal$ be the Zariski closure of $\imGal$ in $\GL_{b,\mathbb Q_\ell}$  and let $\subvar$ be the closed $\ell$-adic analytic subvariety of $\ZarGal (\Q_\ell)$ defined by the equation $\det(M-I)=0$. Define $\subvar(\mathbb Z_\ell):=\subvar \cap \GL_{b}(\mathbb Z_\ell)$.
 If $M\in \imGal$ does not satisfy $\det(M-I)=0$ we must have $\det_\ell(M-I) \leqslant n$ for some $n$,  thus the kernel of $M-I$ is finite: this shows $\imGal\setminus \mathcal M\subseteq \subvar(\mathbb Z_\ell)$. The numerator of \eqref{eq:HaarSmall} is then at most $\# \pi_n(\subvar(\mathbb Z_\ell))$, which by \cite[Theorem 4]{MR656627} is bounded from above by a constant times $\ell^{n \dim(\subvar)}$. To conclude, we only need to prove $\dim(\subvar)<\dim \imGal$.

Suppose instead $\dim(\subvar)=\dim \imGal$. Then $\subvar(\mathbb Z_\ell)$ contains an open subset of $\imGal$ and hence the preimage of $\subvar(\mathbb Z_\ell)$ in $\Gal(\overline{K}/K)$ contains some open subset $U$.
Since Frobenius elements are dense in $\Gal(\overline{K}/K)$, we can find infinitely many of them in $U$ (and by definition any such automorphism acts on $T_\ell(A)$ with a fixed point). We now show that this is impossible.

By Lemma \ref{lem:UnramifiedAlmostEverywhere} we can find a prime $\mathfrak{p}$ of $K$ such that a corresponding Frobenius element is in $U$, the characteristic of $\mathfrak{p}$ is different from $\ell$, and the $\ell$-adic Galois representation attached to $A$ is unramified at $\mathfrak{p}$. Consider the completion $K_\mathfrak p$. 
The assumption that the representation is unramified implies that the image of $\operatorname{Gal}(\overline{K_\mathfrak p}/K_\mathfrak p)$ in $\Aut(T_\ell(A))$ is topologically generated by a Frobenius element, and hence this Galois group acts on $T_\ell(A)$ with a fixed point. This contradicts  the finiteness of the torsion subgroup of $A(K_\mathfrak p)$, see Proposition \ref{finiteness-Av}.
\end{proof}

 \subsection{The condition on the arboreal representation}

We keep the notation of Section\@ \ref{torsionkummer} and suppose that $(A/K, \alpha, \ell)$ are as in Definition \ref{defi-conditions}, fixing $n_0$ as appropriate.
Recall that we identify $\imGal(n)$ and $\operatorname{Gal}\left(K_n/K \right)$ and that we see $\operatorname{Gal}\left(K_{\alpha,n}/K\right)$ as a subgroup of $A[\ell^n] \rtimes \operatorname{Gal}\left(K_n/K \right)$. Denoting by $\pi_1, \pi_2$ the two natural projections, for each $M \in \imGal(n)$ we define the set 
\begin{equation}\label{Wu}
{\set}_n(M):=\pi_1\circ\pi_2^{-1}(M)=\{ t\in A[\ell^{n}] \mid (t, M)\in \Gal({K_{n}(\ell^{-n}\alpha)/K)}\}\,.
\end{equation}

\begin{lem}\label{translate}
The set $\set_n(M)$ is a translate of $\set_n(I)$.
\end{lem}
\begin{proof}
Fix $t_0\in \set_n(M)$. If $t\in \set_n(M)$, we have $(t, M) (t_0, M)^{-1}=(t-t_0,I)$ and hence $t-t_0\in \set_n(I)$. If $v \in \set_n(I)$, we have $(v, I) (t_0, M)=  (v+t_0, M)$ and hence $v+t_0\in \set_n(M)$.
\end{proof}
We also define the rational number  
\begin{equation}\label{wu}
\function_n(M):=\frac{\# \big(\im(M-I) \cap \set_n(M)\big)}{\# \im(M-I)}\,.
\end{equation}

For $M\in \imGal$ we can define $\set_n(M):=\set_n(M_n)$ and $\function_n(M):=\function_n(M_n)$, where $M_n$ is the reduction of $M$ modulo $\ell^n$.

\begin{lem}\label{lemma:Extendw}If $M\in \imGal$, the value $\function_n(M_{n})$ is independent of $n$ for $n \geqslant n_0$, and we call it $\function(M)$: it is either zero or a power of $\ell$ with exponent $\leqslant 0$.
\end{lem}

\begin{proof}
We know $ \im(M_{n_0}-I)=\ell^{n-n_0} \im(M_{n}-I)$ because the following diagram is commutative: 
\begin{center}
\makebox{\xymatrix{
A[\ell^n] \ar[r]^{M_{n}} \ar[d]_{\ell^{n-n_0}} & A[\ell^n] \ar[d]^{\ell^{n-n_0}} \\ A[\ell^{n_0}] \ar[r]^{M_{n_0}} & A[\ell^{n_0}] }}
\end{center}
We also have $\set_n(M_n)=\ell^{-(n-n_0)}\set_{n_0}(M_{n_0})$: an inclusion clearly holds, and the two sets have cardinality $\ell^{b(n-n_0)} \#\set_{n_0}(M_{n_0})$ by \eqref{C2ii} and because by definition we have 
$$\set_{n_0}(M_{n_0})=\{ t\in A[\ell^{n_0}] \mid (t, M_{n_0})\in \Gal({K_{n_0}(\ell^{-n_0}\alpha)/K)}\}$$
where by \eqref{C2i} the condition on $t$ can be rewritten as $(t, M_n)\in \Gal(K_{n}(\ell^{-n_0}\alpha)/K)$.

Denote by $Z$ the kernel of the well-defined and surjective group homomorphism 
\begin{equation}\label{multi-proof}
\ell^{n-n_0} : \im(M_n-I) \to \im(M_{n_0}-I)\,.
\end{equation}
To prove the first assertion it suffices to show that the induced (well-defined and surjective) group homomorphism
\begin{equation}\label{inducedmap}
\ell^{n-n_0} : \im(M_n-I) \cap \set_n(M_n)  \to   \im(M_{n_0}-I) \cap \set_{n_0}(M_{n_0}) 
\end{equation}
is $\# Z$-to-1: this amounts to remarking that if $x\in \im(M_n-I) \cap \set_n(M_n)$ then 
we have $x+Z \subseteq \set_n(M_n)$ because
$\ell^{-(n-n_0)}(\ell^{n-n_0} x)\subseteq \ell^{-(n-n_0)}\set_{n_0}(M_{n_0})= \set_n(M_n)$.

We now fix some $n\geqslant n_0$ and prove that $\function_n(M_n)$ is either zero or a power of $\ell$ (the condition on the exponent is clear from $\function_n(M_n) \leqslant 1$). Recall that $\im(M_n-I)$ and $\set_n(I)$ are finite $\ell$-groups. We may suppose that $\im(M_n-I) \cap \set_n(M_n)$ is non-empty and fix some element $t_0$.
By Lemma \ref{translate} we have
\[
\set_n(M_n) \cap \im(M_n-I) = \big(t_0 + \set_n(I) \big) \cap \big(t_0 + \im(M_n-I)\big)= t_0 + \big(\set_n(I) \cap \im(M_n-I)\big),
\]
which implies our claim since $\set_n(I) \cap \im(M_n-I)$ is an $\ell$-group.
\end{proof}

\begin{exa}
Even if $M\in\mathcal{M}$ we can have $\function(M)=0$. Let $E/\mathbb{Q}$ be a non-CM elliptic curve and $\alpha \in E(\mathbb{Q})$ a point of infinite order such that the arboreal representation attached to $(E/\mathbb{Q}, \alpha, \ell)$ is surjective (for an example, see Section \ref{sect:Examples}), so that its image mod $\ell$ is  $(\Z/\ell\Z)^2\rtimes \GL_2(\Z/\ell\Z)$ (here we have fixed an isomorphism $E[\ell] \cong (\mathbb{Z}/\ell\mathbb{Z})^2$). Consider the cyclic subgroup $H$ 
of $E[\ell] \rtimes \Aut(E[\ell])$
generated by $\left( \begin{pmatrix}
0 \\ 1
\end{pmatrix}, \begin{pmatrix}
1 & 1 \\ 0 & 1
\end{pmatrix} \right)$, which has order $\ell$. Writing $K$ for the fixed field of $H$, the triple $(E/K, \ell, \alpha)$ clearly satisfies the conditions in Definition \ref{defi-conditions} with $n_0=1$.
We can find $M\in \mathcal M$ such that $M_1=\begin{pmatrix}
1 & 1 \\ 0 & 1
\end{pmatrix}$. 
By construction, the set 
$\set(M_1)$ contains only the element $\begin{pmatrix}
0 \\ 1
\end{pmatrix}$, which is not in $\im(M_1-I)$. This shows $\function(M)=\function_1(M_1)=0$.
\end{exa}

\subsection{The general formula for the density}

By Lemma \ref{lemma:EverythingHasWellDefinedAB}, the disjoint union $\mathcal M:=\cup_{\fine} \mathcal M_{\fine}$ has full measure in $\imGal$, hence the domain of integration in \eqref{general-formula2} may be replaced by $\mathcal M$.

\begin{proof}[Proof of Theorem \ref{thm-general}]
Recalling \eqref{MGH}, we consider the set
\begin{equation}
S_{\fine}= \{\sigma=(t,M)\in \Gal(K_\alpha/K)\, :\; M\in \mathcal M_{\fine} \text{ and } t\in \im (M-I)\}\,.
\end{equation}

To see that the Haar measure of $S_{\fine}$ in $\Gal(K_\alpha/K)$ is well-defined and to compute it, we consider the reduction modulo $\ell^n$ of $\mathcal{M}_{\fine}$ and the set
\begin{equation}
S_{\fine,n}=\{\sigma=(t, M)\in \Gal(K_{\alpha,n}/K)\, :  M \in \mathcal{M}_{\fine}(n) \text{ and } t\in \im(M-I)\}\,.
\end{equation}  
We restrict to $n>\max\{n_0,v_\ell(\exp \fine)\}$, where $n_0$ is as in Definition \ref{defi-conditions}. 
By \eqref{Wu} and \eqref{wu} we have 
\[
\#S_{\fine,n}  = \sum_{M\in \mathcal M_{\fine}(n)}\# \big(\im(M-I) \cap \set_n(M)\big) =\sum_{M\in \mathcal M_{\fine}(n)}\ell^{b{n}-\det_\ell(M-I)} \function_n(M)\,.
\]
From \eqref{defi-F} we deduce
\begin{equation}\label{sides}
\frac{\#S_{\fine,n}}{\# \Gal({K_{\alpha,n}/K)}}=   \frac{1}{\#\imGal(n)} \sum_{M\in \mathcal M_{\fine}(n)}  \failure\cdot \ell^{-\det_\ell(M-I)} \cdot \function_n(M).
\end{equation}

By \eqref{refomulation-Defi} the left hand side of \eqref{sides} is a non-increasing function of $n$, and therefore it admits a limit for $n\rightarrow \infty$, which is  $\mu(S_{\fine})$. We claim that $S_{\fine,n}$ is the image of $S_{\fine}$ in $\Gal(K_{\alpha,n}/K)$.

The set $S_{\fine,n}$ clearly contains the reduction modulo $\ell^n$ of $S_{\fine}$, so we prove the other inclusion. Let $\sigma_n=(t_n, M_n) \in S_{\fine,n}$. The natural map $\Gal(K_{\alpha}/K) \to \Gal(K_{\alpha,n}/K)$ is surjective, so there is an element $(t,M)$ in $\Gal(K_{\alpha}/K)$ that reduces to $\sigma_n$. Since  $n>\exp(\fine)$, we have $\ker(M-I)\simeq \ker(M_n-I)$ and hence $M \in \mathcal M_{\fine}$. We now construct an element of $S_{\fine}$ reducing to $\sigma_n$: 
take $a_n \in A[\ell^n]$ satisfying $(M_n-I)(a_n)=t_n$, and consider a lift $a$ of $a_n$ to $ T_\ell(A)$; we may replace $t$ by $(M-I)a$ because the difference is in $\ell^n T_\ell(A)$ and since $n>n_0$ we know that $\Gal(K_{\alpha}/K)$ contains $\ell^n T_\ell(A)\times \{I\}$.

The right-hand side of \eqref{sides} is an integral over $\mathcal M_{\fine}(n)$ with respect to the normalized counting measure of ${\imGal}(n)$, and the matrices in $\mathcal M_{\fine}$ are exactly the matrices in $\imGal$ whose reduction modulo $\ell^n$ lies in $\mathcal M_{\fine}(n)$. By Lemma \ref{lemma:Extendw}, taking the limit in $n$ gives
\begin{equation}\label{musg}
\mu(S_{\fine})= \int_{\mathcal{M}_{\fine}}  \failure \cdot \ell^{-\det_\ell(x-I)} \cdot \function(x)\,\,  d\mu_{\imGal}(x)\,.
\end{equation}

Consider the natural projection $\pi:\Gal(K_\alpha/K)\rightarrow \Gal(K_\infty/K)$.
By Lemmas \ref{Haar2} and \ref{lemma:EverythingHasWellDefinedAB} the set $S_\alpha$ of \eqref{bis} is the disjoint union of the sets $S_{\fine}=S_\alpha\cap \pi^{-1}(\mathcal M_{\fine})$ up to a set of measure $0$, so we may conclude by  Theorem \ref{thm:interpretation} in the form of \eqref{eq:interpretation}.
\end{proof}

\subsection{Equivalent formulations of \eqref{general-formula2}}
Recall that $\Gal(K_\alpha/K)$ is a subgroup of $T_\ell(A)\rtimes \imGal$ and consider the two projections: the integrating function of \eqref{general-formula2} is then 
$$ M \;  \mapsto\;  \mu_{T_\ell(A)} \big(\im(M-I) \cap \pi_1 \circ \pi_{2}^{-1}(M) \big),$$
where $\mu_{T_\ell(A)}$ is the normalized Haar measure on $T_\ell(A)$. Indeed, calling $M_n$ the reduction of $M$ modulo $\ell^n$, we have by \eqref{Wu} and \eqref{wu}
\[
{\ell^{-\det_\ell(M_n-I)}}\cdot {\function_n(M_n)} = \frac{\#\big( \im(M_n-I) \cap \pi_1 \circ \pi_{2}^{-1}(M_n)\big)}{\# A[\ell^n]}
\]
and  $\mu_{T_\ell(A)}$ is the limit of the normalized counting measures on $T_\ell(A)/\ell^nT_\ell(A)\simeq A[\ell^n]$.

\begin{rem}
For every $M \in \mathcal M_{\fine}$ we have $\ell^{\det_\ell(M-I)}=\#\fine$, so we can rewrite \eqref{general-formula2} as
\begin{equation}\label{general-formula}
\Dens_\ell(\alpha)= \sum_{\fine}\frac {\failure}{\# \fine} \cdot \delta(\fine) \qquad \text{where} \quad \delta(\fine):= \int_{\mathcal{M}_{\fine}} \function(x) \, d\mu_{\imGal}(x)
\,.
\end{equation}
\end{rem}

Furthermore, we may restrict the sum in \eqref{general-formula} to those groups $\fine$ which contain a subgroup isomorphic to $A(K)[\ell^\infty]$ because for all but finitely many primes $\p$ of $K$ the group $A(K)[\ell^\infty]$ injects into the group of local points $A(\mathbb{F}_\mathfrak{p})$.

\begin{exa}\label{heuristics}
Suppose that for all $n\geqslant {1}$ the fields $K_{\alpha,n}$ and $K_{\infty}$ are linearly disjoint over $K_n$, and that $[K_{\alpha,n}:K_n]=\ell^{b\max(n-d,0)}$ holds for some $d \geqslant 0$.  We then have
\begin{equation}\label{infinite-sum}
 \Dens_\ell(\alpha) = \sum_{\fine} \frac{1}{\# \ell^d \fine} \cdot 
\mu(\mathcal M_{\fine})\,.
\end{equation}
Indeed, let $M\in \mathcal{M}_{\fine}(n)$ for some $n > \max(d, v_\ell (\exp\fine))$. We know $\im(M-I)=(\Z/\ell^{n} \Z)^{b}/\fine$ and $\set_n(M)= (\ell^d\Z/\ell^{n} \Z)^{b}$, so by elementary group theory we have
$$\function_n(M)  = \frac{\ell^{b(n-d)} \cdot (\# \ell^d\fine)^{-1}}{\ell^{bn}\cdot (\# \fine)^{-1}} = \frac{ \# \fine}{\ell^{bd}\cdot \# \ell^d \fine}
$$
independently of $n$ and $M$, and we may easily conclude because $\failure=\ell^{bd}$.
The density in \eqref{infinite-sum} equals the ``probability" that the $\ell$-part of $(\alpha \bmod \mathfrak p)$ is trivial, if we assume this to be uniformly distributed. Indeed, $\mu(\mathcal M_{\fine})$ is the ``probability" that the $\ell$-part of the group of local points $A_\mathfrak p(k_\mathfrak p)$ is isomorphic to $\fine$ (where $\fine$ varies over all finite abelian $\ell$-groups); in the group $\fine$ the ``probability" that an element (which is an $\ell^d$-power) is coprime to $\ell$ is exactly $(\#\ell^d \fine)^{-1}$. The generic case corresponds to $d=0$.
\end{exa}

\section{Asymptotic behaviour of the density}\label{sect:asymptotics}

\begin{thm}\label{thm:LimitOfDensity}
Let $A/K$ be the product of an abelian variety and a torus defined over a number field.
There exists a positive constant $c:=c(A/K)$ such that the inequality
\[
\Dens_\ell(\ell^n \alpha) \geqslant 1-\frac{c}{ \ell^{n+1}}
\]
holds for all primes $\ell$, for all integers $n\geqslant  0$, and for all points $\alpha\in A(K)$. In particular, $\Dens_\ell(\ell^n\alpha)$ goes to $1$ for $\ell^n \rightarrow \infty$ independently of $\alpha$.
\end{thm}

\begin{lem}\label{homotheties}
Let $A$ be the product of an abelian variety $A'$ and a torus defined over a number field $K$. Call $K'$ the splitting field of the torus. For every prime number $\ell$ and for every integer $n\geqslant 1$ consider the subgroup $H_{\ell,n}$ of $\Gal(K_{\ell^n}/K)$
consisting of the elements that act on $A'[\ell^n]$ as multiplication by some scalar $\lambda$ (if $A'\neq 0$) and that can be lifted to an automorphism of $K'_{\ell^n}/K'$ that acts as exponentiation by $\lambda^2$ on $\zeta_{\ell^n}$. There exists some positive constant $c':=c'(A/K)$ such that for every $\ell$ and $n$ we have $\#H_{\ell,n} \geqslant c'\ell^n$.
\end{lem}

\begin{proof} 
Up to dividing the constant by $[K':K]$ we may work in $\Gal(K'_{\ell^n}/K')$, i.e.~we may suppose that the torus is split and $K=K'$. 
Up to decreasing the constant, we can ignore finitely many prime numbers, so we may suppose that $\ell\neq 2$ and that $K(\zeta_{\ell^n})$ has degree $(\ell-1)\ell^{n-1}$ over $K$.
If $A'$ is trivial and hence $A=\mathbb{G}_m^r$, the $\ell^n$-torsion field is $K(\zeta_{\ell^n})$, so $H_{\ell,n}$ is the group of squares in $\imGal(n) =\operatorname{Gal}(K(\mu_{\ell^n})/K)$, which has order $\frac{\ell-1}{2}\cdot \ell^{n-1}$.
If $A'\neq 0$, call $\hat{A'}$ the dual abelian variety of $A'$ and let $S_{\ell,n}$ be the subgroup of 
$\imGal(n)$ consisting of those elements that act as a scalar on $(A' \times \hat{A'})[\ell^n]$. Notice that $A' \times \hat{A'}$ depends only on $A$.

By a theorem of Serre-Wintenberger (\cite[Théorème 3]{MR1944805}) there is a constant $d:=d(A/K)$ such that for every $\ell$ and for every $k \in \mathbb{Z}^\times$ the matrix $k^d \cdot I$ is in the image of the $\ell$-adic representation attached to $A' \times \hat{A'}$, so we have $\# S_{\ell,n} \geqslant \frac{\ell^n}{4d}$ (the index of the subgroup of $d$-th powers in $(\Z/\ell^n\Z)^\times$ is at most $2d$ and $\#(\Z/\ell^n\Z)^\times \geqslant \frac{1}{2}\ell^n$). 
Considering the Weil pairing 
$A'[\ell^n]\times \hat{A'}[\ell^n]\rightarrow \langle \zeta_{\ell^n}\rangle$
we deduce that $S_{\ell,n}$ is contained in $H_{\ell,n}$ and we are done. \end{proof}

\begin{lem}\label{lem:UsuallyInvertibleDeterminant}
Let $A/K$ be the product of an abelian variety and a torus defined over a number field.
For any integer $n\geqslant 1$ consider the set
\[
B_{\ell,n} := \{x \in \imGal \bigm\vert \mathrm{det}_\ell(x-I) \geqslant n\}\,.
\]
There exists a constant $c:=c(A/K)$ such that $\mu(B_{\ell,n}) \leqslant c \ell^{-n}$ holds for all $\ell$ and $n$.
\end{lem}

\begin{proof}
Set $B_{\ell,n}(n):=\{ M \in \imGal : \det(M) \equiv 1 \pmod{\ell^n} \}$, and keep the notation of Lemma \ref{homotheties}. Write $\imGal(n) = \coprod_{r \in \mathcal{R}} H_{\ell,n} \cdot r$, where $\mathcal{R}$ is a set of representatives for the cosets of $H_{\ell,n}$ in $\imGal(n)$. We identify an element of $H_{\ell,n}$ to its corresponding scalar $\lambda$. For a given $r$, the quantity $\det(\lambda r)=\lambda^b \det(r)$ is congruent to 1 modulo $\ell^n$ if and only if $\lambda^b \equiv \det(r)^{-1} \pmod {\ell^n}$. 
For any fixed $r$, at most $2b$ values of $\lambda$ satisfy this congruence, and therefore every coset contains at most $2b$ matrices in $B_{\ell,n}(n)$.
From Lemma  \ref{homotheties} we deduce
\[
\#B_{\ell,n}(n) \leqslant 2b \cdot \frac{\#\imGal(n)}{\#H_{\ell,n}} \leqslant \frac{2b}{c'} \ell^{-n} \cdot \#\imGal(n)\,.
\]
Since $B_{\ell,n}$ is the inverse image in $\imGal$ of $B_{\ell,n}(n)$ we get   
$\mu(B_{\ell,n}) = \frac{\# B_{\ell,n}(n)}{\#\imGal(n)} \leqslant \frac{2b}{c'} \cdot \ell^{-n}$.
\end{proof}

\begin{proof}[Proof of Theorem \ref{thm:LimitOfDensity}]
We want to compute the density of the set of primes $\p$ for which the $\ell$-adic valuation of the order of $(\alpha \bmod \p)$ is at most $n$. The complement of this set consists of primes $\p$ for which $\#A(\mathbb{F}_{\mathfrak{p}})$ is divisible by $\ell^{n+1}$; in particular, the Frobenius at $\mathfrak{p}$ is an element of the set $B_{\ell,n+1}$ of Lemma \ref{lem:UsuallyInvertibleDeterminant}. The result follows immediately.
\end{proof}

\section{The density for elliptic curves}\label{EC}

\subsection{Computability of the density}\label{compu-EC}
We show that, in the special case of $A/K$ being an elliptic curve, the value $\Dens_\ell(\alpha)$ can be effectively computed for all $\alpha\in A(K)$. 

One can first determine whether the order of $\alpha$ is either coprime to $\ell$, divisible by $\ell$, or $\infty$. In the first two cases $\Dens_\ell(\alpha)$ is respectively $1$ and $0$, so we can assume without loss of generality that $\alpha$ is a point of infinite order. In this case, the set $\mathbb{Z}\alpha$ is Zariski-dense in $A$, so by Remark \ref{nieuw} we know that $\operatorname{Dens}_\ell(\alpha)$ is given by \eqref{general-formula}. We shall make use of the following definition:

\begin{defi}
A subset of $\mathbb N^2$ is \emph{admissible} if it is the product of two subsets of $\mathbb N$ which are either finite or consist of all integers greater than some given one. The family of finite unions of admissible sets is closed with respect to intersection, union and complement.

\end{defi}

As in Section \ref{sect:1Eigen} we study subsets of $\imGal$ of the form $\mathcal{M}_{\fine}$, where $\fine$ is a finite subgroup of $A[\ell^\infty]$. Since $A$ is an elliptic curve we can write $\fine=\mathbb{Z}/\ell^a\mathbb{Z} \times \mathbb{Z}/\ell^{a+b}\mathbb{Z}$ for some integers $a,b\geqslant 0$. Setting $\mathcal M_{a,b}:=\mathcal M_{\fine}$  and 
\begin{equation}\label{deltab}
\delta(a,b):= \frac{1}{\mu(\mathcal{M}_{a,b})}\int_{\mathcal{M}_{a,b}} \function(x) \, d\mu_{\imGal}(x),
\end{equation}
Equation \eqref{general-formula} becomes
\begin{equation}\label{general-formula-ab}
\Dens_\ell(\alpha)=\failure \cdot \sum_{(a,b)\in \mathbb N^2}  \mu(\mathcal M_{a,b}) \cdot \ell^{-2a-b} \cdot \delta(a,b)\,. \end{equation}

\begin{prop}\label{prop:ComputeDeltas}
The set $\mathbb N^2$ can be partitioned in finitely many (effectively computable) admissible sets such that on each of them \eqref{deltab} is independent of $(a,b)$ and is an effectively computable rational number.
\end{prop}
\begin{proof}
We first compute an integer $n_0$ as in Definition \ref{defi-conditions}, see Remark \ref{rem:n0ComputableForEllipticCurves}. We may suppose $\mathcal{M}_{a,b}\neq \emptyset$ because by \cite[Theorem 1]{LombardoPerucca1Eig} 
the pairs $(a,b)$ satisfying $\mathcal{M}_{a,b} = \emptyset$ form an explicitly computable admissible set, and  for them $\delta(a,b)=0$. By definition, we know
\begin{equation}\label{deltas}
\delta(a,b)=\lim_{n\rightarrow \infty} f_{a,b}(n) \qquad f_{a,b}(n):=\frac{1}{\#\mathcal{M}_{a,b}(n)} \sum_{M \in \mathcal{M}_{a,b}(n)} \function(M)\,.
\end{equation}
Any single $\delta(a,b)$ can be computed effectively because $f_{a,b}(n)$ is independent of $n$ for $n>\max\{n_0,a+b\}$. Indeed, we have $f_{a,b}(n+1)=f_{a,b}(n)$ because   any lift of $M \in \mathcal{M}_{a,b}(n)$ to $\imGal(n+1)$ belongs to $\mathcal{M}_{a,b}(n+1)$ and hence all matrices in $\mathcal{M}_{a,b}(n)$ have the same number of lifts to $\mathcal{M}_{a,b}(n+1)$, namely $\#\mathbb{T}$. We also use the fact (Lemma \ref{lemma:Extendw}) that $\function(M)$ only depends on $M$ modulo $\ell^{n_0}$.
We shall repeatedly use the following fact: if we have
\begin{equation}\label{uniformity}
\# \{ M \in \mathcal{M}_{a,b}(n) : M \equiv M_0 \pmod{\ell^{n_0}} \}=\frac{\# \mathcal{M}_{a,b}(n)}{\# \mathcal{M}_{a,b}(n_0)} \quad \forall M_0 \in \mathcal{M}_{a,b}(n_0)
\end{equation}
for some $n>\max\{n_0,a+b\}$ then by Lemma \ref{lemma:Extendw} we also have 
\begin{equation}\label{delta-finite}
\delta(a,b)=f_{a,b}(n_0)= \frac{1}{\#\mathcal{M}_{a,b}(n_0)} \sum_{M_0 \in \mathcal{M}_{a,b}(n_0)} \function(M_0)\,.
\end{equation}

\emph{If $\imGal$ is open either in $\GL_2(\Z_\ell)$ or in the normalizer of a split/nonsplit Cartan:} By \cite[Theorem 28]{LombardoPerucca1Eig} we know that \eqref{uniformity} holds  and hence $\delta(a,b)=f_{a,b}(n_0)$.
Since $\mathcal{M}_{a,b}(n_0)\neq\emptyset$, we get 
\[
\delta(a,b) = \left\{ \begin{array}{ll}
  1  & \mbox{if $a \geqslant n_0$}\\
\delta(a,n_0-a)  & \mbox{if $a+b \geqslant n_0>a$}
\end{array} \right.
\]
because for $a \geqslant n_0$ the only matrix in $\mathcal{M}_{a,b}(n_0)$ is the identity and $\function(I)=1$, while for $b \geqslant n_0-a>0$ the sets $\mathcal{M}_{a,b}(n_0)$ and $\mathcal{M}_{a,n_0-a}(n_0)$ coincide by \cite[Proposition 32]{LombardoPerucca1Eig}. The assertion easily follows.

\emph{If $\imGal$ is open in the normalizer $N$ of a Cartan subgroup $C$ which is neither split nor nonsplit:} We suppose $n_0\geqslant 2$, and let $(0,d)$ be the parameters of $C$, see \cite[Section 2.3]{LombardoPerucca1Eig}.  
Recall that finitely many pairs $(a,b)$ can be treated individually, so we restrict to $a+b>n_0$ and have $\delta(a,b)=f_{a,b}(a+b+1)$.

Write $C_{a,b}:=\mathcal M_{a,b}\cap C$ and $C'_{a,b}:=\mathcal M_{a,b}\cap (N\setminus C)$, and recall from \cite[Proposition 26]{LombardoPerucca1Eig} that ${C}'_{a,b}= \emptyset$ for $a\geqslant 1$ if $\ell$ is odd and for $a \geqslant 2$ if $\ell=2$.
By \cite[Theorems 27 and 28]{LombardoPerucca1Eig}, a necessary condition for \eqref{uniformity} not to hold is that both $C_{a,b}(n_0)$ and $C'_{a,b}(n_0)$ are non-empty, and hence $a=0$ if $\ell$ is odd and $a \in \{0,1\}$ if $\ell=2$.

\emph{The case when $d$ is not a square in $\Z_\ell^\times$.} By \cite[Lemma 37]{LombardoPerucca1Eig} we know that for any fixed $a$ the set $C_{a,b}$ is empty for $b$ sufficiently large (and the result is effective).
In particular \eqref{uniformity} may fail only for finitely many and explicitly computable pairs $(a,b)$, which we may individually consider. So we may suppose $\delta(a,b)=f_{a,b}(n_0)$.

Provided that $C_{a,b}$ is empty, $f_{a,b}(n_0)$ is independent of $b$ for $b > n_0$: this follows from \cite[Theorem 31 (ii)]{LombardoPerucca1Eig} because in this reference the set $\mathcal{N}_{a,b}(n_0)$ has this property. 
Moreover, if $a\geqslant n_0$ then $f_{a,b}(n_0)=1$ because $\mathcal{M}_{a,b}(n_0)=\{I\}$. The assertion easily follows.

\emph{The case when $d$ is a square in $\Z_\ell$.} For $a \geqslant 2$ we have $C'_{a,b} = \emptyset$, thus \eqref{uniformity} and hence \eqref{delta-finite} hold. By \cite[Lemmas 37 (iii) and 38]{LombardoPerucca1Eig} 
we can deal with this case as above, so suppose $a\in\{0,1\}$.
The number of lifts to $\mathcal{M}_{a,b}(n+1)$ of a matrix $M$ in $\mathcal{M}_{a,b}(n)$ depends at most on whether $M$ belongs to the trivial/nontrivial coset of $C$ in $N$ (see \cite[Theorem 28]{LombardoPerucca1Eig}), so we have:
\begin{equation}\label{cc'}
\sum_{M \in \mathcal{M}_{a,b}(n)} \function(M)  = \sum_{M_0 \in C_{a,b}(n_0)} \function(M_0) \, \frac{\# C_{a,b}(n)}{\# C_{a,b}(n_0)}\; + \sum_{M_0 \in C'_{a,b}(n_0)} \function(M_0) \, \frac{\# C'_{a,b}(n)}{\# C'_{a,b}(n_0)}.
\end{equation}

For $n$ sufficiently large ($n>\max\{n_0,a+b\}$ suffices) we have $\#{C}'_{a,b}(n) / \#\mathcal{M}_{a,b}(n) =  \mu({C}'_{a,b}) / \mu (\mathcal M_{a,b})$ and by \cite[Corollary 41]{LombardoPerucca1Eig} this equals some constant $c_a$ for all sufficiently large $b$ (the bound is effective). So by \eqref{cc'} the following holds for all sufficiently large $b$:
$$\delta(a,b) = \sum_{M_0 \in C_{a,b}(n_0)} \function(M_0) \, \frac{1-c_a}{\# C_{a,b}(n_0)} \;+ \sum_{M_0 \in C'_{a,b}(n_0)} \function(M_0) \, \frac{c_a}{\# C'_{a,b}(n_0)}.
$$
For any fixed $a$, the sets $C_{a,b}(n_0)$ and $C'_{a,b}(n_0)$ are independent of $b$ for $b$ large enough (and the bound is effective), see \cite[Lemma 39]{LombardoPerucca1Eig} and \cite[Theorem 31 (ii)]{LombardoPerucca1Eig}. We deduce that $\delta(0,b)$ and $\delta(1,b)$ are constant  for all sufficiently large $b$, with an effective bound. The assertion easily follows.
\end{proof}

\begin{proof}[Proof of Theorem \ref{compu-thm}]
We show that the right-hand side of \eqref{general-formula-ab} is an effectively computable rational number whose (minimal) denominator satisfies the property given in the statement. By Lemma \ref{lemma:DefinitionF}, the constant $\failure$ is defined in terms of $n_0$ (effectively computable by Remark \ref{rem:n0ComputableForEllipticCurves})
and $\#\Gal(K_{\alpha,n_0}/K_{n_0})$. Since this Galois group is computable, the constant $\failure$ is an explicitly computable power of $\ell$ and we are left to investigate the sum in \eqref{general-formula-ab}.

By \cite[Theorem 1]{LombardoPerucca1Eig} and Proposition \ref{prop:ComputeDeltas}
we can partition $\mathbb N^2$ into finitely many (explicitly computable) admissible sets $S$ such that for each of them there is an  (explicitly computable) rational constant $c_S\geqslant 0$ satisfying
$$\mu(\mathcal M_{a,b})\cdot \delta(a,b) \cdot \ell^{-2a-b}=  c_S\cdot \ell^{-Ea-b}$$
for every $(a,b)\in S$. The sum in \eqref{general-formula-ab}, restricted to the pairs $(a,b)\in S$, then becomes
$$c_S\cdot \sum_{a\in S_1} \ell^{-Ea} \cdot \sum_{b\in S_2} \ell^{-2b}$$
where the sets $S_1,S_2$ are finite or have a finite complement in $\mathbb N$. We can explicitly evaluate the  geometric series, and each sum is a  rational number whose denominator divides a power of $\ell$ times $\ell^{E}-1$ or $\ell^{2}-1$. We conclude by proving that, up to powers of $\ell$, the minimal denominator of $c_S$ divides $\#\imGal(n)$ for some $n\geqslant 1$ (this is enough to establish the proposition since $\#\imGal(n)$ divides $\#\GL_2(\mathbb Z/\ell\mathbb Z)=\ell(\ell^2-1)(\ell -1)$ up to a power of $\ell$). Consider \cite[Lemma 25]{LombardoPerucca1Eig}, formula \eqref{deltas} and the assertion following it: if $c_S\neq 0$ we can fix $(a,b)\in S$ and take $n$ sufficiently large so that we have
$$\mu(\mathcal M_{a,b})\cdot \delta(a,b)=\mu(\mathcal M_{a,b}(n)) \cdot f_{a,b}(n)=
\frac{\#\mathcal M_{a,b}(n)}{\#\imGal(n)}\cdot \frac{1}{\#\mathcal M_{a,b}(n)}\sum_{M\in \mathcal M_{a,b}(n)} \function(M)\,.$$
Since $\mathcal M_{a,b}(n)$ is a finite set and $\function(M)$ is a power of $\ell$, we are done.
\end{proof}

\subsection{Surjective arboreal representations}\label{surjective}

The following result generalizes \cite[Theorems 5.5 and 5.10]{JonesRouse}, which correspond to the special case $d=0$. The expression $1 - {\ell^{1-d}}/({\ell^2-1})$ is the density for the multiplicative group, see \cite[Theorem 1]{PeruccaKummer}.

\begin{thm}\label{thm-surj}
Let $A/K$ be an elliptic curve, and let $\alpha\in A(K)$ be a point of infinite order. Fix a prime number $\ell$. 
Suppose that for all $N>n\geqslant 1$ the fields $K_{\alpha,n}$ and $K_{N}$ are linearly disjoint over $K_{n}$. Also suppose that there is some integer $d\geqslant 0$ satisfying $[K_{\alpha,n}:K_n]=\ell^{2\max(n-d,0)}$ for every $n\geqslant 1$.
\begin{enumerate}
\item If the image of the $\ell$-adic representation attached to $A$ is $\GL_2(\mathbb Z_{\ell})$, we have:
$$\Dens_\ell(\alpha)= 1 - \frac{\ell^{1-d} \cdot (\ell^3-\ell-1)}{(\ell^2-1)\cdot(\ell^3-1)}\,.$$

\item If the image of the $\ell$-adic representation attached to $A$ is either a split or a nonsplit  Cartan subgroup of $\GL_2(\mathbb Z_{\ell})$, we respectively have:
$$\Dens_\ell(\alpha)= \Big(1 - \frac{\ell^{1-d}}{\ell^2-1}\Big)^2 \qquad \Dens_\ell(\alpha)= 1-\frac{\ell^{2(1-d)}}{\ell^4-1}\,.$$

\item If the image of the $\ell$-adic representation attached to $A$ is the normalizer of a Cartan subgroup of $\GL_2(\mathbb Z_{\ell})$ which is split or nonsplit, we have:
$$\Dens_\ell(\alpha)=\frac{1}{2}\cdot \Big(1 - \frac{\ell^{1-d}}{\ell^2-1}\Big)+\frac{1}{2}\cdot 
\left\{ \begin{array}{ll}
\Big(1 - \dfrac{\ell^{1-d}}{\ell^2-1}\Big)^2  & \mbox{for a split Cartan}\medskip \\
\Big(1-\dfrac{\ell^{2(1-d)}}{\ell^4-1}\Big)  & \mbox{for a nonsplit Cartan}\,.\\
\end{array} \right.$$
\end{enumerate}
\end{thm}

\begin{proof}
Writing a closed formula for \eqref{infinite-sum} amounts to evaluating some simple geometric series, because in \cite[Section 1]{LombardoPerucca1Eig} we have explicit formulas for the measures $\mu(\mathcal M_{a,b})$.
\end{proof}

One can easily write analogous parametric formulas for the simultaneous reductions of many points: for $i=1,\ldots, n$ consider elliptic curves $A_i/K$ and points $\alpha_i\in A(K)$ of infinite order. The density of primes $\mathfrak p$ such that the order of $(\alpha_i \bmod \mathfrak p)$ is coprime to $\ell$ for every $i$ is exactly the density $\Dens_\ell(\alpha)$ for the point $\alpha=(\alpha_1,\ldots,\alpha_n)$ in the product $\prod_i A_i$.

\subsection{Examples for Theorem \ref{thm-surj}}\label{sect:Examples}
We tested the formulas of Theorem \ref{thm-surj} in the examples below: the exact value of $\Dens_\ell(\alpha)$ was always in excellent agreement with a  numerical approximation computed with SAGE (by restricting to primes up to $10^5$).

For the non-CM elliptic curve $E: y^{2} + y = x^{3} - x$ over $\Q$ and the point $\gamma=(0,0)$, the $\ell$-arboreal representation is surjective onto $T_\ell(E) \rtimes \GL_2(\mathbb Z_{\ell})$ for every $\ell$ \cite[Example 5.4]{JonesRouse} and we have:
\begin{center}
\begin{tabular}{|c||ccc|ccc|cc|cc|}
\hline
$\ell$ & \multicolumn{3}{c|}{2}  & \multicolumn{3}{c|}{3} & \multicolumn{2}{c|}{5} & \multicolumn{2}{c|}{7}\\
\hline
$\alpha$ & $\gamma$ & $2\gamma$ & $4\gamma$ & $\gamma$ & $3\gamma$ & $9\gamma$ & $\gamma$ & $5\gamma$ & $\gamma$ & $7\gamma$\\
\hline
\multirow{2}{*}{$\Dens_\ell(\alpha)$} & \multirow{2}{*}{$\dfrac{11}{21}$}  & \multirow{2}{*}{$\dfrac{16}{21}$} & \multirow{2}{*}{$\dfrac{37}{42}$}  & \multirow{2}{*}{$\dfrac{139}{208}$} & \multirow{2}{*}{$\dfrac{185}{208}$}  & \multirow{2}{*}{$\dfrac{601}{624}$} & \multirow{2}{*}{$\dfrac{2381}{2976}$} & \multirow{2}{*}{$\dfrac{2857}{2976}$} & \multirow{2}{*}{$\dfrac{14071}{16416}$} & \multirow{2}{*}{$\dfrac{16081}{16416}$}\\
&&&&&&&&&&\\
\hline
\end{tabular}
\end{center}

For the CM elliptic curve $E: y^{2} = x^{3} + 3x$ over $\Q$ and the point $\gamma=(1,-2)$, 
the $5$-adic representation is surjective onto the normalizer of a split Cartan subgroup of $\GL_2(\Z_5)$, and the Kummer extensions are as large as possible \cite[Example 5.11]{JonesRouse}. We then have $\Dens_5(\gamma)=817/1152$ and $\Dens_5(5\gamma)=1081/1152$.

For the CM elliptic curve $E: y^{2} = x^{3} + 3$ over $\Q$ and the point $\gamma=(1,2)$, 
the $2$-adic representation is surjective onto the normalizer of a nonsplit Cartan subgroup of $\GL_2(\Z_2)$, and the Kummer extensions are as large as possible  \cite[Example 5.12]{JonesRouse}. So we have 
$\Dens_2(\gamma)=8/15$, $\Dens_2(2\gamma)=4/5$ and 
$\Dens_2(4\gamma)=109/120$.

\subsection{Example (non-surjective mod 3 representation)}\label{sub-exa1}
Consider the non-CM elliptic curve $E: y^2 + y = x^3 + 6x + 27$ over $\mathbb{Q}$  \cite[\href{http://www.lmfdb.org/EllipticCurve/Q/153/b/2}{label 153.b2}]{lmfdb} and the point of infinite order $\alpha=(5,13)$. The image $\imGal$ of the 3-adic representation is open in $\GL_2(\Z_3)$ and we have $$
\imGal(1)=\langle \begin{pmatrix}
1 & 1 \\ 0 & 1
\end{pmatrix}, \begin{pmatrix}
-1 & 0 \\ 0 & 1
\end{pmatrix}\rangle
$$
so that $\imGal(1)$ is a subgroup of $\GL_2(\Z/3\Z)$ of order 6.
The $9$-division polynomial of $E$ has an irreducible factor of degree 27 whose splitting field has degree $2 \cdot 3^5$. Since the $9$-division polynomial factors completely over $\mathbb{Q}(E[9])$, we deduce
\[
2 \cdot 3^5 \bigm\vert [\mathbb{Q}(E[9]):\mathbb{Q}]=  [\mathbb{Q}(E[9]):\mathbb{Q}(E[3])]\cdot [\mathbb{Q}(E[3]):\mathbb{Q}]\mid 3^4\cdot 6
\]
and hence we have $[\mathbb{Q}(E[9]):\mathbb{Q}(E[3])]=3^4$. By Theorem \ref{thm:IfXThenLevel}, $\imGal$ is the inverse image of $\imGal(1)$ in $\GL_2(\Z_3)$. We then have $[\GL_2(\Z_{3}):\imGal]=8$, and one can check $\mu(\mathcal M_{0,0}(1))=0$. From \cite[Proposition 32]{LombardoPerucca1Eig} we get $\mu(\mathcal M_{a,b}(1))={1}/{6}$ for $a>0$ and 
$\mu(\mathcal M_{0,b}(1))=5/6$ for $b>0$. By \cite[Proposition 33]{LombardoPerucca1Eig} we then obtain 
\[
\mu(\mathcal M_{a,b})=\left\{
\begin{array}{ll}
0 & \text{ if } a=b=0 \\ 5 \cdot 3^{-b-1}& \text{ if } a=0, b>0 \\ 8 \cdot 3^{-4a}& \text{ if } a>0, b=0 \\ 32 \cdot 3^{-4a-b-1}& \text{ if } a>0, b>0 \,.
\end{array}
\right.
\]

We show below that the image of the 3-arboreal representation of $\alpha$ is $T_3 E \rtimes \imGal$; it follows that $$\operatorname{Dens}_3(\alpha)=\sum_{a,b \geqslant 0} \mu(\mathcal M_{a,b}) 3^{-2a-b}=\frac{23}{104}=0.22115...$$
Considering only primes up to $10^5$, SAGE computes the approximate density $0.22116$.
We similarly have $\operatorname{Dens}_3(3\alpha)=77/104=0.74038...$ and $\operatorname{Dens}_3(9\alpha)=95/104=0.91346...$, for which SAGE gives as approximate densities respectively $0.73806$ ($0.74034$ considering primes up to $10^6$) and $0.91126$.

To prove that the image of the arboreal representation is $T_3E \rtimes \imGal$ we apply Theorem \ref{thm:IfXThenLevel} with $n=1$. We need to verify $[K_{\alpha,2}:K_{\alpha,1}]=3^6$: one computes without difficulty $[K_{\alpha,1}:\mathbb{Q}]=2\cdot 3^3$, and we know $[K_2 : \mathbb{Q}]=2 \cdot 3^5$, so we are left to check $[K_{2,\alpha}:K_2]=3^4$. One divisibility is clear, so let us prove that $3^4$ divides the degree of $K_{2,\alpha}$ over $K_2$. 

Denote by $L$ the field (of degree $3^4$) generated over $\mathbb{Q}$ by a root of the 9-division polynomial of $\alpha$. Since $K_{\alpha,2}\supseteq LK_2$, it suffices to show that $L$ and $K_2$ are linearly disjoint over $\Q$. If not, exploiting the structure of the Galois group $\Gal(K_2/\mathbb{Q})$ we see there would be a subfield of $L \cap K_2$ of degree $3$ over $\mathbb Q$.
However, this field cannot exist because one can test with SAGE that the 9-division polynomial of $\alpha$ is irreducible over all subextensions of $\mathbb{Q}(E[9])$ of degree $3$ over $\Q$.

\subsection{Example (index $3$ in the normalizer of a split Cartan)}\label{sub-exa2}

Consider the elliptic curve $y^2+y=x^3+7140 $ over $\mathbb{Q}$ \cite[\href{http://www.lmfdb.org/EllipticCurve/Q/1521/b/2}{label 1521.b2}]{lmfdb}, which has potential complex multiplication by $\mathbb{Z}[\zeta_3]$. Consider the point of infinite order $\alpha=(56,427)$.
The image $\imGal$ of the 13-adic Galois representation is properly contained in the normalizer of a split Cartan subgroup of $\GL_2(\Z_{13})$: the image $\imGal(1)$ of the modulo-13 representation is generated by the matrices
\[
\begin{pmatrix}
2 & 0 \\ 0 & 2
\end{pmatrix},
\begin{pmatrix}
5 & 0 \\ 0 & 1
\end{pmatrix}, \begin{pmatrix}
0 & 1 \\ -1 & 0
\end{pmatrix}
\]
thus it is a group of order $96$ (and index 3) in the full normalizer of a split Cartan subgroup of $\GL_2(\Z/13\Z)$, and one can check that $\imGal$ is the inverse image of $\imGal(1)$ in $\GL_2(\Z_{13})$. 
We keep the notation from \cite{LombardoPerucca1Eig}. A direct computation gives $\mu^C_{0,0}(1) = 41/48$, $\mu^C_{0,1}(1) = 1/8$ and hence  $\mu_{a,b}>0$ for every $a,b\geqslant 0$: we may then apply \cite[Propositions 32-33]{LombardoPerucca1Eig}.
A direct computation gives $\mu_{0,0}^*(1)=11/12$, $\mu_{0,1}^*(1)=1/12$, and we can apply \cite[Theorem 40]{LombardoPerucca1Eig}. So we have 
\[
\mu^C_{a,b} = \left\{\begin{array}{lllll}
\frac{41}{48} & \text{ if } & a=b=0 \\
\frac{3}{2} \cdot 13^{-b} & \text{ if } & a=0, b \geqslant 1 \\
3 \cdot 13^{-2a} & \text{ if } & a \geqslant 1, b=0 \\
6 \cdot 13^{-2a-b} & \text{ if } & a \geqslant 1, b \geqslant 1
\end{array}\right.
\; \text{ and }\;
\mu^*_{a,b}=\left\{\begin{array}{lllll}
{11}/{12} & \text{ if } & a=0, b = 0 \\
13^{-b} & \text{ if } & a =0 , b \geqslant 1 \\ 0 & \text{ if } & a>0
\end{array}\right.
\]
which by \cite[Section 3.6]{LombardoPerucca1Eig} determines $\mu(\mathcal M_{a,b})=\frac{1}{2} (\mu^{C}_{a,b} + \mu_{a,b}^*)$.
We claim that the image of the 13-arboreal representation of $\alpha$ is $T_{13} E \rtimes \imGal$, so we have 
$$\operatorname{Dens}_{13}(\alpha)=\sum_{a,b \geqslant 0} \mu(\mathcal M_{a,b}) \cdot 13^{-2a-b}=\frac{16801}{18816} =  1-\frac{36270}{(13^2-1)^2(13-1)}  = 0.89291...$$
Considering only primes up to $10^5$, SAGE computes the approximate density $0.89322$.
We similarly have $\operatorname{Dens}_{13}(13\alpha)=1-167/18816=0.99112...$. Considering primes up to $10^6$, SAGE computes the approximate density $0.99131$.

The claim can be proven by applying \cite[Theorem 3.4]{JonesRouse} once we have checked its two assumptions. Firstly, $E[13]$ is an irreducible $\imGal(1)$-module and hence $E[13]^{\imGal(1)}=\{0\}$. We make use of \cite[Lemmas 3.6 and 3.7]{JonesRouse}. 
We know $\alpha \not \in 13 E(\mathbb{Q})$ because $\alpha$ generates $E(\mathbb{Q})/E(\mathbb{Q})_{\tors}$. So we are left to prove that for $n \geqslant 1$ there is no nonzero homomorphism of $\imGal(1)$-modules between $J_n:=\ker(\imGal(n+1) \to \imGal(n))$ (where the action is conjugation) and $E[13]$ (with the usual Galois action). 
Such a homomorphism would be surjective (the image is a non-trivial $\imGal(1)$-submodule of $E[13]$), and hence also injective because $\#J_n=13^2$. This is not possible because $E[13]$ is irreducible and $J_n$ has the 1-dimensional submodule $\langle (1+13^n)I\rangle$.

\section{Universality of denominators} \label{universality}

This Section is devoted to proving Theorem \ref{Conjecture}. For any given dimension there are only finitely many possible values for the first Betti number, so we prove instead:

\begin{thm}\label{conjecture-b}
Fix $b\geqslant 1$. There exists a polynomial $p_b(t)$ such that whenever $K$ is a number field and $A/K$ is the product of an abelian variety and a torus with first Betti number $b$, then for all prime numbers $\ell$ and for all $\alpha \in A(K)$ we have $\Dens_\ell(\alpha)\cdot p_b(\ell)  \in \Z[1/\ell]$.
\end{thm}

\subsection{Preliminaries}

Fix an algebraic subgroup $\algQ$ of $\GL_{b,\Q_\ell}$. We set $\algQ(\mathbb{Z}_\ell):=\algQ(\mathbb{Q}_\ell) \cap \GL_{b}(\Z_\ell)$ and define $\algRid(n)$ as the reduction modulo $\ell^n$ of $\algQ(\mathbb{Z}_\ell)$. There is a rational constant $c(\algQ)$ such that $\#\algRid(n) = c(\algQ){\ell^{n\cdot\dim(\algQ)}}$ holds for all  sufficiently large $n$. We also introduce the polynomial
$$gl_b(t):=\prod_{k=0}^{b-1} (t^b-t^k),$$ 
which satisfies $\#\GL_b(\mathbb Z/\ell\mathbb Z)=gl_b(\ell)$ for all primes $\ell$.

\begin{lem}\label{lemma:cG}
We have $c(\algQ)^{-1} \cdot gl_b(\ell) \in \Z[1/\ell]$ and $\#\algRid(n)^{-1} gl_b(\ell) \in \Z[1/\ell]$ for every $n \geqslant 1$.
\end{lem}
\begin{proof}
For every sufficiently large $n$ we have $c(\algQ)^{-1} \cdot gl_b(\ell)=\ell^{nd} \cdot \#\algRid(n)^{-1} \cdot gl_b(\ell)$
and we may conclude because $\#\algRid(n)$ divides $\#\GL_{b}(\Z/\ell^n\Z) = gl_b(\ell) \cdot \ell^{b^2 (n-1)}$.
\end{proof}

\subsection{Generating sets of polynomials}
Let $k$ be any field of characteristic $0$ (we will only need the result for ${\mathbb Q_\ell}$).

\begin{defi}\label{def:ClassC}
We say that an algebraic subgroup $\algQ$ of $\operatorname{GL}_{b,k}$ is of class $\mathcal{C}$ if it is reductive, connected, and if the weights of the radical of $\algQ_{\overline{k}}$ acting on ${\overline{k}}{}^b$ (via the tautological representation) are in $\{0,1\}$. Recall that the radical is an algebraic torus given by the identity component of the center of $\algQ$.
\end{defi}

\begin{rem}\label{rem-classC}
Let $\algQ \subseteq \operatorname{GL}_{b,k}$ be a group of class $\mathcal{C}$.
Then $\algQ_{\overline{k}}$ is of class $\mathcal{C}$ (the group $\algQ$ is geometrically connected because it is connected and has a rational point \cite[Tag 04KV]{stacks-project}). 
The radical of $\algQ$ is of class $\mathcal{C}$ (it is its own radical, hence the weights are the same as for $\algQ$). The derived subgroup of $\algQ$ is of class $\mathcal{C}$ (its radical is trivial, hence the weights of its action are zero).
\end{rem}

\begin{lem}\label{prop:SemisimplePlusCenter}
There exists an integer $D(b)$ such that any group $\algQ_{\bar{k}}\subseteq \operatorname{GL}_{b,\bar{k}}$ of class $\mathcal{C}$ can be defined in $\bar{k}[x_{ij},y]/(\det(x_{ij})y-1)$ by finitely many polynomials of degree at most $D(b)$.
\end{lem}
\begin{proof}
We know that $\algQ_{\overline{k}}$ is the almost-direct product of its derived group $S$ (which is semi-simple) and of its radical $T$. Notice that a change of basis in $\GL_{b,\bar{k}}$ changes neither the degree nor the number of the equations defining a subgroup, so we can work up to conjugation. 

We first prove the statement for $S$. By the Lefschetz principle, we may work over $\mathbb C$ and it suffices to remark that there are only finitely many conjugacy classes of semi-simple Lie subalgebras of $\mathrm{Lie}(\operatorname{GL}_{b,\mathbb C})$ \cite[Section 12 (a)]{Richardson}.

The statement is also true for $T$. Indeed, up to conjugation, the action of $T$ on $\overline{k}{}^b$ is given by 
$$(z_1,\ldots,z_r) \mapsto \operatorname{diag}( z_1^{a_{11}} \cdots z_r^{a_{1r}}, \ldots, z_1^{a_{b1}} \cdots z_r^{a_{br}} )$$ where the exponents $a_{ij}$ are the weights of the tautological representation and hence are in $\{0,1\}$ by assumption. The equations defining $T$ are $x_{ij}=0$ for $i \neq j$ and the finitely many equations of bounded degree $\prod_{i=1}^b x_{ii}^{v_i}=1$, where the vector $(v_i)$ ranges over a basis of $\ker (a_{ij})$.

To conclude, we make use of the theory of bounded complexity. We have shown that $S$ and $T$ have bounded complexity and hence the same holds for $S\times T$ \cite[Definition 3.1 and Remarks]{MR2827010}.
Since the product map $\GL_{b,\overline{k}} \times \GL_{b,\overline{k}} \to \GL_{b,\overline{k}}$ has bounded complexity \cite[Definition 3.3]{MR2827010}, the same holds  by \cite[Lemma 3.4]{MR2827010}  for the restriction to $S \times T$ and for its image $\algQ_{\overline{k}}$.
\end{proof}

\begin{thm}\label{thm:FinitelyManyGroups}
Let $\algQ \subseteq \operatorname{GL}_{b,k}$ be a group of class $\mathcal{C}$. There exist integers $D(b), N(b)$ (depending only on $b$) such that $\algQ$ can be defined in $R := k[x_{ij},y]/(\det(x_{ij})y-1)$ by at most $N(b)$ polynomials of degree at most $D(b)$.
\end{thm}

\begin{proof}
It suffices to show that there are defining polynomials of degree at most $D(b)$ because the polynomials in $k[x_{ij},y]$ of a given degree form a finite dimensional vector space (whose dimension can be bounded purely in terms of the number of variables and of the degree, hence ultimately in terms of $b$).
We let $I$ be the ideal of $\algQ$ in $R$ and $I_{\overline{k}}=I \otimes \overline{k}$ the ideal of $\algQ_{\overline{k}}$ in $R \otimes \overline{k}$.
By Remark \ref{rem-classC} and Lemma \ref{prop:SemisimplePlusCenter}, the ideal  $I_{\overline{k}}$ is defined by finitely many polynomials $f_i$ of degree at most $D(b)$. Fix a finite, Galois extension $L$ of $k$ that contains all their coefficients, and let $\{t_j\}$ be a basis of $L$ over $k$. The polynomials $\operatorname{tr}_{L/k}\left(t_j f_i \right)$
(the trace is taken coefficientwise) are in $I$ and have degree bounded by $D(b)$. They generate $I$ because we can write $f \in I$ as 
\[
f = \sum_{i=1}^r a_i \cdot f_i = \frac{1}{[L:k]} \sum_{i=1}^r  \operatorname{tr}_{L/k}(a_i \cdot f_i)
\]
for some polynomials $a_i$ in $R \otimes L$.
\end{proof}

\subsection{The image of the $\ell$-adic representation}

Let $A$ be the product of an abelian variety and a torus defined over a number field. Let $\ZarGal$ be the Zariski closure in $\GL_{b,\Q_\ell}$ of the image of the $\ell$-adic Galois representation attached to $A$.

\begin{prop}\label{prop:ClassC}
The identity component of $\ZarGal$ is of class $\mathcal{C}$.
\end{prop}
\begin{proof}
The group $\ZarGalId$ is clearly connected, and it is the product of the identity components of the $\ell$-adic monodromy groups associated with the torus and with the abelian variety. The claim is obvious for the torus, because we can reduce to the case of $\mathbb{G}_m$, which is clear (the weight of the tautological representation is $1$).
We may then assume that $A$ is an abelian variety, and hence $\ZarGalId$ is reductive by a celebrated theorem of Faltings. We are left to understand the tautological representation $\rho$ of the reductive group $(\ZarGalId)_{\overline{\mathbb{Q}_\ell}}$ on $\overline{\mathbb{Q}_\ell}{}^b$. Since $\rho$ is the direct sum of irreducible representations, we may consider the weight of every irreducible factor $\rho'$ separately. The weights of the action of the radical are in $\{0,1\}$ by the discussion following \cite[Definition 4.1]{MR1603865}: notice that by \cite[Theorem 5.10]{MR1603865} the pair given by $\ZarGalId$ together with its tautological representation is a weak Mumford-Tate pair in the sense of \cite[Definition 4.1]{MR1603865}.
\end{proof}

\begin{prop}\label{prop:ConnectedComponents}
There is a non-zero integer $z(b)$, depending only on the Betti number $b$, such that number of connected components of $\ZarGal$ divides $z(b)$.
\end{prop}
\begin{proof}
For every prime number $p$ let  $\rho_p : \Gal(\overline{K}/K) \to \GL_b(\Z_p)$ be the $p$-adic representation attached to $A$, and call $\ZarGal_{,p}$ the Zariski closure of its image. Let $K^{\text{conn}}$ be the finite extension of $K$ corresponding to $\rho_p^{-1}\left(\ZarGalId_{,p}(\Q_p) \cap \GL_b(\Z_p)\right)$. The degree $[K^{\text{conn}}: K]$ is the number of connected components of $\ZarGal_{,p}$.
It is known by work of Serre \cite{Serre_resum8485} (cf.\@ also \cite[Introduction]{MR1441234}) that $K^{\text{conn}}$ is independent of $p$, so the degree $[K^{\text{conn}}:K]$ divides the greatest common divisor of the supernatural numbers $\#\GL_b(\Z_p)=p^{\infty} \cdot \#\GL_b(\mathbb{F}_p)$, which is an integer depending only on $b$.
\end{proof}

\subsection{A theorem of Macintyre}
We apply a result of Macintyre \cite{Macintyre} which -- roughly speaking -- is a uniformity statement for integrals over $L_{P}$-definable sets. $L_{P}$ is a first-order language (in the sense of logic) that is similar to Denef's language used in \cite{MR902824}. It is obtained from the language of rings (i.e.~a language with $0, 1, +, \cdot, -$) by adding 1-ary predicates $P_n$, for $n \geqslant 2$. We make $\mathbb{Q}_\ell$ an $L_P$-structure by interpreting $P_n$ as the set of $n$-th powers in $\mathbb{Q}_\ell$ (i.e.~$P_n(x)$ is true iff $x$ is an $n$-th power in $\Q_\ell$). The valuation ring $\Z_\ell = \{ x : v_\ell(x) \geqslant 0\}$ is $L_P$-definable:

\begin{lem}\label{lemma:ValuationRing}
There is a formula $\Phi(x)$ in $L_P$ (independent of $\ell$) such that, when $\Q_\ell$ is interpreted as an $L_P$-structure as above, we have $\Phi(x) \Leftrightarrow v_\ell(x) \geqslant 0$.
\end{lem}
\begin{proof}
Expanding the argument in \cite[p.71]{Macintyre}, consider the formula $R(x,y)$ given by $P_2(1+y^3x^4)$ and let
$V := \{ y : R(x,y) \text{ defines a valuation ring} \}$.
The property $y \in V$ is expressible by a formula $\Omega(y)$ in $L_P$ because the property of defining a valuation ring can be expressed in the language of rings.
Indeed, $S\subseteq \Q_\ell$ is a valuation ring if and only if it is a subring that satisfies $(\forall x)(x \in S \text{ or } \exists x' : xx'=1, x' \in S)$.
When we interpret $\Q_\ell$ as an $L_P$-structure, $R(x,-\ell)$ defines precisely $\Z_\ell$. One can check that $v_\ell(x) \geqslant 0$ if and 
 only if $(\forall y \in V) R(x,y)$: it follows that $v_\ell(x) \geqslant 0$ is equivalent to $(\forall y)( \Omega(y) \Rightarrow R(x,y))$.
\end{proof}

\begin{cor}\label{cor:ValuationsInLP} There is an $L_P$-formula $\Psi(x,y)$ such that, when $\Q_\ell$ is interpreted as an $L_P$-structure, we have $\Psi(x,y) \Leftrightarrow v_\ell(x) \geqslant  v_\ell(y)$. Likewise, there are formulas that encode the statements $v_\ell(x)=0$, $v_\ell(x) > v_\ell(y)$, and $v_\ell(x)=v_\ell(y)+1$.
\end{cor}

\begin{proof}
We can express $v_\ell(x) \geqslant v_\ell(y)$ by $\exists z : \Phi(z) \wedge (x=yz)$ and $v_\ell(x)=0$ by $\exists y : \Phi(x) \wedge \Phi(y) \wedge (xy=1)$.
The property $v_\ell(z) > 0$ means $v_\ell(z) \geqslant 0 \wedge v_\ell(z) \neq 0$ while $v_\ell(x) > v_\ell(y)$ means $\exists z : v_\ell(z)>0, x=yz$.
Finally, $v_\ell(x)=v_\ell(y)+1$ is equivalent to $v_\ell(x) > v_\ell(y)$ and $(\forall z)(v_\ell(z) > 0 \Rightarrow v_\ell(yz) \geqslant v_\ell(x))$.
\end{proof}

Let $\Phi$ be a formula in $L_{P}$ with $m+m'$ free variables (such a formula can be interpreted in $\mathbb{Q}_\ell$ for every prime $\ell$). Define 
\begin{equation}\label{lambda}
\mathcal A:=\{ (X,\lambda) \in \mathbb{Q}_\ell^m \times \mathbb{Q}_\ell^{m'} : \Phi(X, \lambda) \}\qquad \text{and}\qquad \mathcal A(\lambda):=\{ X  : (X,\lambda) \in \mathcal A\} \subseteq \Q_\ell^m\,.
\end{equation}
Such a set $\mathcal{A}$ is said to be $L_P$-definable. 
We also consider functions $\mathbb{Q}_\ell^n\rightarrow \mathbb{Z} \cup \{+\infty\}$. 
We deal only with \emph{$L_{P}$-simple} functions of the form $v_\ell(f(x_1,\ldots,x_{n}))$ where $f \in \mathbb{Z}[x_1,\ldots,x_{n}]$ (see \cite[Definition after Lemma 18]{Macintyre}, replacing $L_{PD}$ by $L_P$).

\begin{thm}{(Macintyre \cite{Macintyre})}\label{thm:Mcintyre}
Suppose $\mathcal A$ is an $L_{P}$-definable set, and $\alpha, \alpha'$ are $L_{P}$-simple functions.   We have
\[
\int_{\mathcal A(\lambda)} \ell^{-\alpha(X,\lambda)s-\alpha'(X,\lambda)} \, dX = \frac{\sum_{1 \leqslant i,i' \leqslant \varepsilon(\lambda)} \gamma_{i,i'} \, \ell^{-is-i'}}{c \prod_{1\leqslant j<h} \big(1-\ell^{-a_js-a'_j}\big)}
\]
whenever the integral is finite, where 
\begin{enumerate}
\item $\varepsilon$ is $L_P$-simple with values in $\mathbb{N}$, and the $\gamma_{i,i'}$ are integers;
\item $h$ is a constant independent of $\ell$; 
\item the $a_j, a'_j$ are natural numbers, bounded by some constant $\tau$  independent of $\ell$;
\item $c$ divides $(\ell(\ell-1))^m$, where $m$ is the dimension of the integration space;
\item the numbers $a_j, a'_j, \tau, h$ and $c$, as well as the function $\varepsilon$, only depend on the formula defining $\mathcal{A}$ and on the polynomials defining $\alpha, \alpha'$ (in particular, they are independent of $\lambda$ and $\ell$).
\end{enumerate}
\end{thm}
Though not stated in this exact form, this theorem is fully proved in \cite{Macintyre}: the main result is Corollary 2 on p.70, (2) is a direct consequence of Theorem 19 of op.cit., (3) is proved in §7.2.1, and (4) is proved in §7.2.2.

\subsection{Rationality of $\ell$-adic integrals}
Take  $N:=N(b),D:=D(b)$ as in Theorem \ref{thm:FinitelyManyGroups} and define the set 
\[
\mathcal{P}:=\operatorname{Mat}_b(\Z_\ell) \times \left( \Z_\ell[x_{ij}]_{1 \leqslant i,j \leqslant b, \deg \leqslant D} \right)^N \times \Z_\ell\,.
\]
An element of $\mathcal{P}$ will be written as
$\lambda:=(T;f_1,\ldots,f_N; w_n)$ 
where $T \in \operatorname{Mat}_b(\Z_\ell)$, $f_1$ to $f_N$ are polynomials in $\Z_\ell[x_{ij}]$ (for $1 \leqslant i,j \leqslant b$) of degree at most $D$, and $w_n$ is an $\ell$-adic integer which we only use through its valuation. For $z\in \mathbb Z_\ell$ we call reduction modulo $z$ the identity map if $z=0$ and the reduction modulo $\ell^{v_\ell(z)}$ otherwise.
The set 
\begin{equation}\label{defi-A}
\resizebox{1\hsize}{!}{
$\mathcal A=\left \{ (x, w; \lambda) \in \operatorname{Mat}_b(\Z_\ell) \times \Z_\ell \times \mathcal{P} \bigm\vert \exists M \in \operatorname{Mat}_b(\Z_\ell) : \!\!\!\!\!\!\begin{array}{c} f_1(M)=\cdots=f_N(M)=0 \\ v_\ell(\det(M)) = 0\\ M \equiv x \pmod{w} \\ M \equiv I \pmod{w_n} \\ v_\ell(\det(TM-I))+1=v_\ell(w)  \end{array} \right \}
$}
\end{equation}
 is $L_{P}$-definable: indeed, the vanishing of $f_i(M)$ is simply given by the vanishing of a suitable family of polynomials in the entries of $M$ and of $\lambda$ (this can be expressed in the language of rings), and the congruence conditions can be expressed in $L_P$ by Lemma \ref{lemma:ValuationRing} and Corollary \ref{cor:ValuationsInLP}. We define $\mathcal A(\lambda)$ as in \eqref{lambda}.

\begin{prop}\label{prop:IntegralIsUniversal}
Fix $b\geqslant 1$. For every algebraic subgroup $\algQ$ of $\GL_{b,\Q_\ell}$ of class $\mathcal{C}$, for every integer $n\geqslant 0$, and for every $T \in \GL_b(\Z_\ell)$, there exists $\lambda \in \mathcal{P}$ such that $\mathcal A(\lambda)$ equals
 \begin{equation}\label{eq:SetD}
\mathcal D_{n}:=\left\{ (x,w) \in \operatorname{Mat}_b(\Z_\ell) \times \Z_\ell \bigm\vert \exists M \in \algQ(\Z_\ell) : \begin{array}{c} M \equiv x \pmod{w} \\  M \equiv I \pmod{\ell^n} \\ \det_\ell(TM-I)+1=v_\ell(w) \end{array} \right\}\,.
 \end{equation}
 \end{prop}
\begin{proof}
By Theorem \ref{thm:FinitelyManyGroups}, there exist polynomials $f_1,\ldots,f_N$ of degree at most $D$ which define $\algQ$ in $\GL_{b,\Q_\ell}$. The conditions $M \in \operatorname{Mat}_b(\Z_\ell)$, $v_\ell(\det(M))=0$ and  $f_1(M)=\ldots=f_N(M)=0$ give $M \in \algQ(\Z_\ell)$. The other conditions for $\mathcal A(\lambda)$ match  those of $\mathcal{D}_n$ if we take $w_n=\ell^n$.
\end{proof}

\begin{defi} Let $\mathcal D_{n}$ be as in \eqref{eq:SetD}. We set
$
\displaystyle I_{n}(s) := \int_{\mathcal D_{n}} |w|^s \, dx \, dw\,.
$
\end{defi}

\begin{lem}\label{lemma:UniversalDenominatorRationalFunction}
Fix $b\geqslant 1$. There exists a polynomial $r(t,u) \in \Z[t,u]$ such that for every $\ell$, for every algebraic subgroup $\algQ$ of $\GL_{b,\Q_\ell}$ of class $\mathcal{C}$, for every $T \in \GL_b(\Z_\ell)$, for every integer $n\geqslant 0$ and for every real number $s>0$ we have 
\[
I_{n}(s) = \frac{\Psi_n (\ell^{-s})}{r(\ell^{-s},\ell^{-1})}\,,
\]
where $\Psi_n(t)$ is a polynomial in $\mathbb{Z}[1/\ell][t]$ which may depend on $n,\ell,\algQ,T$.
\end{lem}

\begin{proof} By Proposition \ref{prop:IntegralIsUniversal} there exists $\lambda$ such that $\mathcal{D}_n=\mathcal A(\lambda)$, and the integral $I_n(s)$ is finite because $|w|^s\leqslant 1$ and the integration space has finite measure.
Theorem \ref{thm:Mcintyre} (choosing $\mathcal A$ as in \eqref{defi-A}, 
$\alpha= v_\ell(w)$ and $\alpha'=0$) then gives 
\[
I_{n}(s) = \frac{\Psi(\ell^{-s})}{c \prod_{1\leqslant j<h} \big(1-\ell^{-a_js-{a'_j}}\big)}\,,
\]
where we have set $\Psi(t):=\sum_{1 \leqslant i,i' \leqslant \varepsilon(\lambda)} \gamma_{i,i'} \ell^{-i'} t^{i} \in \Z[1/\ell][t]$. Again by Theorem \ref{thm:Mcintyre}, the denominator divides $r(\ell^{-s},\ell^{-1})$ up to a power of $\ell$, where $$r(t,u) := (1-u)^{b^2+1} \prod_{0 \leqslant a, a' \leqslant \tau} (1-t^{a}u^{a'})^h\,.$$
Notice that we can reabsorb the power of $\ell$ in the numerator. To finish the proof, recall that $\tau$ and $h$ only depend on the polynomial defining $\alpha$ (which is in particular independent of $\ell$) and on the formula defining $\mathcal A$. Since by construction the latter depends only on $b$, this establishes our claim.
\end{proof}

\subsection{Uniform bound for the denominators}
Let $\algQ$ be an algebraic subgroup of $\GL_{b,\Q_\ell}$ of class $\mathcal{C}$.
Fix some matrix $T$ in $\GL_{b}(\Z_\ell)$. If $n,n'\geqslant 0$  and $m\geqslant 1$ are integers,
we define 
\begin{equation}\label{eq:SetS}
S_{n,n'}(m) :=  \left\{ M \bmod {\ell^m} \bigm\vert 
\begin{array}{c} M \in \algQ(\Z_{\ell}) \\ M \equiv I \pmod{\ell^n} \\ \det_\ell(TM-I) = n' \end{array}
 \right\}
 \end{equation}
 and $N_{n,k}:=\# S_{n,k-1}(k)$ (we also set $N_{n,0}=0$).
We consider the Poincar\'e series
 \begin{equation}\label{Poinc}
 P_{n}(t) := \sum_{k \geqslant 0} N_{n,k} \, t^k = \frac{\ell}{\ell-1} \cdot \frac{\Psi_{n}(t \ell^{b^2+1})}{r(t \ell^{b^2+1}, \ell^{-1})}
\end{equation}
where the second equality follows from Lemma \ref{lemma:UniversalDenominatorRationalFunction} and the following computation:
\[
\begin{aligned}
I_{n}(s) & = \sum_{k=0}^\infty \ell^{-ks} \int_{\mathcal D_n \cap \{v_\ell(w)=k\}} dx \, dw \\
& = \sum_{k=0}^\infty \ell^{-ks} \left( \int_{x : (x,\ell^k) \in \mathcal D_n} dx \right) \left( \int_{w : v_\ell(w)=k}  dw \right) \\
& = \sum_{k=0}^\infty \ell^{-ks} \cdot \frac{N_{n,k}}{\ell^{kb^2}} \cdot \left( \ell^{-k}-\ell^{-k-1} \right) \\
& = \frac{\ell-1}{\ell} P_{n}(\ell^{-(s+1+b^2)}).
\end{aligned}
\]

\begin{lem}\label{cool}
Fix $b\geqslant 1$, and let $r(t,u)$ be as in Lemma \ref{lemma:UniversalDenominatorRationalFunction}. For every integer $n_0\geqslant 0$ we have:
\[
(\ell-1) \cdot r(\ell^{b^2-d},\ell^{-1})\cdot \sum_{k \geqslant n_0} \ell^{-dk-d-k} N_{n,k+1} \in \mathbb{Z}[1/\ell]\,.
\]
\end{lem}

\begin{proof}
We can write
\[
\sum_{k \geqslant n_0} \ell^{-dk-d-k} N_{n,k+1}
= \ell \cdot P_n(\ell^{-(d+1)}) - \ell \cdot \sum_{0 \leqslant i \leqslant n_0} \ell^{-i(d+1)} N_{n,i}.
\]
The finite sum is obviously in $\Z[1/\ell]$. We may conclude by applying \eqref{Poinc}: 
$$\ell \cdot P_n(\ell^{-(d+1)})\cdot (\ell-1) \cdot r(\ell^{b^2-d},\ell^{-1}) =\ell^2 \cdot \Psi_n(\ell^{b^2-d})\in \Z[1/\ell]\,.$$
\end{proof}

\begin{prop}\label{prop:IntegralReformulation}
Fix $b\geqslant 1$. There exists a polynomial $q_b(t) \in \Z[t]$ with the following property. If $\algQ$ is an algebraic subgroup of $\GL_{b,\Q_\ell}$ of class $\mathcal{C}$, $n\geqslant 0$ is an integer, and $T$ is any matrix in $\GL_{b}(\Z_\ell)$, we have:
\[
q_b(\ell) \cdot \int_{\{ M \in \algQ(\Z_\ell) \bigm\vert \; M \equiv I \,(\!\bmod{\ell^n}) \}} \ell^{-\det_\ell(T M-I)} \, d\mu_{\algQ(\Z_\ell)}(M) \; \in \mathbb{Z}[1/\ell].
\]
\end{prop}

\begin{proof}
Define $q_b(t):= gl_b(t) \cdot (t-1) \cdot r(t^{b^2-d},t^{-1}) t^{\deg_u r(t,u)}$, where $r(t,u)$ is as in Lemma \ref{lemma:UniversalDenominatorRationalFunction}.
Notice that if $m > \max(n,n')$ the set
\[
S_{n,n'}:=\left\{ M \bigm\vert M \in \algQ(\Z_\ell), M \equiv I\, (\bmod{\ell^n}), \mathrm{det}_\ell(TM-I)=n' \right\}
\]
is the inverse image in $\algQ(\Z_\ell)$ of the set $S_{n,n'}(m) \subseteq \algRid(m)$ from \eqref{eq:SetS}. In particular, if $n_0 > n$ is sufficiently large and $k \geqslant n_0$ we can write 
\[
\mu_{\algQ(\Z_\ell)}(S_{n,k} ) =\frac{\#S_{n,k}(k+1)}{\#\algRid(k+1)} = \frac{N_{n,k+1}}{c(\algQ) \ell^{d(k+1)}}.
\]
For $k<n_0$ we may write instead $\mu_{\algQ(\Z_\ell)} \left(S_{n,k} \right)={\#S_{n,k}(n_0)}/{\#\algRid(n_0)}$. We then have
$$ 
\resizebox{1\hsize}{!}{
$\displaystyle \int_{\{ M : M \equiv I\, (\!\bmod{\ell^n}) \}}  \ell^{-\det_\ell(T M-I)} \, d \mu_{\algQ(\Z_\ell)}(M) = \sum_{k \geqslant n_0} \ell^{-k} \frac{N_{n,k+1}}{c(\algQ) \ell^{d(k+1)}} + \sum_{k < n_0} \ell^{-k}\frac{\#S_{n,k}(n_0)}{\#\algRid(n_0)}\,.$
}
$$
The second sum, when multiplied by $q_b(\ell)$, gives an element of $\Z[1/\ell]$ by  Lemma \ref{lemma:cG}. For the first sum we may apply Lemmas \ref{lemma:cG} and \ref{cool}.
\end{proof}

\begin{proof}[Proof of Theorem \ref{conjecture-b}]
We may assume without loss of generality that the order of $\alpha$ is infinite, because otherwise $\Dens_\ell(\alpha)\in \{0,1\}$.
By Remark \ref{nieuw} and Theorem \ref{thm-general}, $\Dens_\ell(\alpha)$ is an integral multiple of 
\begin{equation}\label{F-dens}
\int_{\imGal} \function(x) \, \ell^{-\det_\ell(x-I)} \, d\mu_{\imGal}(x)\,,
\end{equation}
where $\imGal \subseteq \GL_b(\Z_\ell)$ is open in its Zariski closure $\ZarGal$ by Proposition \ref{prop:GeneralizedBogomolov}. If $n$ is sufficiently large, we can write $\imGal$ as the disjoint union of finitely many cosets $T_i \neighbourhood$, where 
\[
\neighbourhood := \{ x \in \mathcal \ZarGalId(\Z_\ell) : x \equiv I \pmod{\ell^{n}} \}\,.
\]
By Lemma \ref{lemma:Extendw}, we may take $n$ sufficiently large so that $\function$ is constant on $T_i\neighbourhood$. We then rewrite \eqref{F-dens} as the finite sum of terms of the form
\[
\int_{T_i\neighbourhood}  \function(T_i) \,  \ell^{-\det_\ell(x-I)} \, d\mu_{\imGal}(x)
= \frac{[\ZarGal(\Z_\ell) : \imGal ]}{[\ZarGal:\ZarGalId]} \cdot \function(T_i)\cdot \int_{\neighbourhood} \ell^{-\det_\ell(T_i x-I)} \, d\mu_{\ZarGalId(\mathbb{Z}_\ell)}(x).
\]
Since $\function(T_i)\in \Z[1/\ell]$ by Lemma \ref{lemma:Extendw}, we may conclude by taking  $p_b(t)=z(b)q_b(t)$, where $z(b)$ is as in Proposition \ref{prop:ConnectedComponents} and $q_b(t) \in \mathbb{Z}[t]$ is as in Proposition \ref{prop:IntegralReformulation}.
\end{proof}

\begin{proof}[Proof of Theorem \ref{thm:rationality}]
Given a group $\imGal$ (the image of the $\ell$-adic Galois representation attached to $A$) and its Zariski closure $\ZarGal$, we can find integers $N$ and $D$ with the property that $\ZarGal$ is defined by at most $N$ polynomials of degree at most $D$. One can now repeat verbatim the proof of Theorem \ref{conjecture-b}, replacing $N(b)$ and $D(b)$ with these $N$ and $D$.
\end{proof}

\bibliographystyle{abbrv}
\bibliography{biblio}

\end{document}